\documentclass{amsart}

\usepackage{amssymb,amsmath,amscd,graphicx,amsthm,color,enumerate}

\usepackage{pst-node}
\usepackage{tikz-cd}

\newtheorem{theorem}{Theorem}
\newtheorem{lemma}[theorem]{Lemma}
\newtheorem{proposition}[theorem]{Proposition}

\theoremstyle{definition}

\theoremstyle{remark}
\newtheorem{remark}[theorem]{Remark}

\numberwithin{equation}{section}
\numberwithin{theorem}{section}

\def\A{{\mathcal A}}
\def\AA{{\mathbb A}}

\def\C{{\mathbb C}}
\def\CC{{\mathcal C}}

\def\F{{\mathbb F}}
\def\FF{{\mathcal F}}
\def\FFF{{\mathcal F}}
\def\G{{\mathcal G}}
\def\GCC{{\G\CC}}
\def\NGCC{\overline{\GCC}}

\def\L{{\mathcal L}}

\def\O{{\mathcal O}}
\def\P{{\mathcal P}}

\def\T{{\mathbb T}}
\def\TA{{\mathcal T}}

\def\UU{\overline{\A}}

\def\X{{\mathcal X}}
\def\Y{{\mathcal Y}}
\def\Z{{\mathbb Z}}

\def\b{\mathfrak b}

\def\fy{\varphi}

\def\gl{\mathfrak g\mathfrak l}

\def\n{\mathfrak n}
\def\one{\mathbf 1}

\def\phhi{{\varphi}}

\def\q{{\bf q}}

\def\wx{{\widetilde{\bf x}}}
\def\x{{\bf x}}

\def\Gr{\operatorname{Gr}}

\def\Mat{\operatorname{Mat}}

\def\Poi{{\{\cdot,\cdot\}}}

\def\Tr{\operatorname{Tr}}

\def\diag{\operatorname{diag}}

\def\rank{\operatorname{rank}}

\def\sign{{\operatorname{sign}}}

\def\:{{:\ }}

\begin{document}

\title[Generalized Cluster Structures Related to the Double of $GL_n$]
{Generalized Cluster Structures Related to the Drinfeld Double of $GL_n$}

\author{Misha Gekhtman}
\address{Department of Mathematics, University of Notre Dame, Notre Dame,
IN 46556}
\email{mgekhtma@nd.edu}

\author{Michael Shapiro}
\address{Department of Mathematics, Michigan State University, East Lansing,
MI 48823, and National Resarch University Higher School of Economics, Russia}
\email{mshapiro@math.msu.edu}

\author{Alek Vainshtein}
\address{Department of Mathematics \& Department of Computer Science, University of Haifa, Haifa,
Mount Carmel 31905, Israel}
\email{alek@cs.haifa.ac.il}

\begin{abstract}
We prove that the regular generalized cluster structure on the Drinfeld double of $GL_n$ constructed in \cite{GSVstaircase}
is complete and compatible with the standard Poisson--Lie structure on the double. Moreover, we show that for $n=4$ this
structure is distinct from a previously known regular generalized cluster structure on the Drinfeld double, 
even though they have the same compatible Poisson structure and the same collection of frozen variables. Further, we prove that
the regular generalized cluster structure on band periodic matrices constructed in \cite{GSVstaircase} possesses similar 
compatibility and completeness properties.
\end{abstract}

\maketitle

\medskip

\section{Introduction}

 It is by now well-known that many important algebraic varieties arising in Lie theory, representation theory and theory of integrable systems support a cluster structure. The first example of this kind is already present in the foundational paper \cite{FZ2} where it was shown that the homogeneous coordinate ring  of the Grassmannian of $2$-planes in $\mathbb{C}^{n+3}$ is naturally isomorphic to the cluster algebra of finite type $A_n$. Among the examples that followed were Grassmannians \cite{GSVb, Scott}, double Bruhat cells \cite{CAIII} and strata in flag varieties \cite{Leclerc}. All of these examples share two key features: (i) the variety in question is equipped with a Poisson brackets {\em compatible} with the cluster structure in a sense reviewed in Section 2.1 below, and (ii) cluster transformations that connect distinguished coordinate charts within a ring of regular functions are modeled on three-term relations such as short Pl\"ucker relations, Desnanot--Jacobi identities and their Lie-theoretic generalizations. The first feature led us to development of  an approach for constructing a cluster structure in Poisson varieties possessing a particular nice coordinate chart (see, e.g.,  \cite{GSVb}). However, there are situations when reliance on three-term relations (equivalently, usual cluster transformations) turns out to be too restrictive and when certain multinomial versions of cluster transformation are needed. These were  first considered in \cite{CS} and termed {\em generalized cluster transformations}. 
The first geometric example of this sort was studied in \cite{GSVCR,GSVDouble} where we used a more general form of transformations defined in \cite{CS} to construct an initial seed $\Sigma_n$ for a complete generalized cluster structure $\GCC_n^D$  in the standard Drinfeld double  $D(GL_n)$ 
 and proved that this structure is compatible with the standard Poisson--Lie structure on $D(GL_n)$.

In \cite[Section 4]{GSVstaircase} we presented a rich source of identities that can serve as generalized cluster transformation and, as one of the applications, constructed a different seed $\bar\Sigma_n$ for a regular generalized cluster structure $\NGCC_n^D$ on $D(GL_n)$. 
In this paper we prove that $\NGCC_n^D$ shares all the properties of $\GCC_n^D$: it is complete and
compatible with the standard Poisson--Lie structure on $D(GL_n)$.  Moreover, we prove that the 
seeds $\bar\Sigma_4(X,Y)$ and $\Sigma_4(Y^T,X^T)$ are not mutationally equivalent. This answers the question posed
by S. Keel: "Do there exist two different regular cluster structures on the same variety with
the same compatible Poisson bracket and the same collection of frozen variables?"
by providing an explicit example of two different regular complete generalized cluster structures on $D(GL_4)$ 
with the same compatible Poisson structure and the same collection of frozen variables.
Further, from the above properties of $\NGCC_n^D$ we derive that the generalized cluster structure in the ring of regular functions on band periodic matrices built in \cite[Section 5]{GSVstaircase} is complete and compatible with the restriction of the standard Poisson--Lie structure.  Apart from possible applications to cluster integrable systems, the latter generalized cluster structure is closely related to the conjectural ones  in cyclic symmetry loci in Grassmannians considered in the recent preprint \cite{FraserNew} that appeared while this paper was under review, and in the Grothendieck rings of the quantum affine algebras $U_q\widehat{\mathfrak{sl}_k}$ at roots of unity  \cite{Gleitz}. These connections will be explored in a joint work of M.G. with C.~Fraser and K.~Trampel.

Section \ref{prelim} below contains all necessary information about generalized cluster structures borrowed mainly from 
\cite{GSVstaircase} to make this text self-contained. Section \ref{doublestruct} is devoted to the study of $\NGCC_n^D$.
The initial seed is described in Section \ref{init}. The main result of this section is Theorem \ref{newstructure} which
claims that $\NGCC_n^D$ is  compatible with the standard Poisson--Lie structure on $D(GL_n)$ and complete. The former statement
is proved in Sections \ref{logcan} and \ref{compat}, and the latter in Section \ref{complet}. In Section \ref{twogcs}  we
compare two generalized cluster structures on $D(GL_4)$: $\NGCC_4^D$ and $\GCC_4^D(Y^T,X^T)$ described
in \cite{GSVDouble}. These two structures have the same set of frozen variables, and we prove that they are distinct, that is, 
the two seeds are not mutationally equivalent. Finally, Section \ref{bandmat} treats the case of periodic band matrices. 
The initial seed $\Sigma_{kn}$ for the
generalized cluster structure on the space $\L_{kn}$ of $(k+1)$-diagonal $n$-periodic band matrices 
is described in Section \ref{bandinit}. The main result of this section is 
Theorem \ref{Band_structure} which claims that $\GCC(\Sigma_{kn})$ is compatible with the restriction of the 
standard Poisson--Lie structure on $D(\Mat_n)$ and complete. The former statement is proved in Section \ref{bandcompat}, and
the latter in Section \ref{bandcomplet}.

\section{Preliminaries} \label{prelim}
\subsection{Generalized cluster structures}
Following \cite{GSVDouble}, we remind the definition of a generalized cluster structure represented by a quiver 
with multiplicities. Let $(Q,d_1,\dots,d_N)$ be
a quiver on $N$ mutable and $M$ frozen vertices with positive integer multiplicities $d_i$ at mutable vertices. 
A vertex is called {\it special\/} if its multiplicity is greater than~1. A frozen vertex is called {\it isolated\/}
if it is not connected to any other vertices. Let $\F$ be the field of rational functions in $N+M$ independent variables
with rational coefficients. There are $M$  distinguished variables corresponding to frozen vertices; 
they are denoted $x_{N+1},\dots,x_{N+M}$ and called {\em stable}, or {\em frozen\/} variables. The {\it coefficient group\/} is a free multiplicative abelian group of Laurent monomials in stable variables, 
and its integer group ring is $\bar\AA=\Z[x_{N+1}^{\pm1},\dots,x_{N+M}^{\pm1}]$ (we write
$x^{\pm1}$ instead of $x,x^{-1}$).

An {\em extended seed\/} (of {\em geometric type\/}) in $\F$ is a triple
$\Sigma=(\x,Q,\P)$, where $\x=(x_1,\dots,x_N, x_{N+1},\dots, x_{N+M})$ is a transcendence basis of $\F$ over the field of
fractions of  $\bar\AA$ and $\P$ is a set of $N$ {\em strings}. The $i$th string is a collection of 
monomials $p_{ir}\in\AA=\Z[x_{N+1},\dots,x_{N+M}]$, $0\le r\le d_i$, such that  
$p_{i0}=p_{id_i}=1$; it is called {\em trivial\/} if $d_i=1$, and hence both elements of the string are equal to one.
The monomials $p_{ir}$ are called {\em exchange coefficients}.

Given a seed as above, the {\em adjacent cluster\/} in direction $k$, $1\le k\le N$,
is defined by $\x'=(\x\setminus\{x_k\})\cup\{x'_k\}$,
where the new cluster variable $x'_k$ is given by the {\em generalized exchange relation}
\begin{equation}\label{exchange}
x_kx'_k=\sum_{r=0}^{d_k}p_{kr}u_{k;>}^r v_{k;>}^{[r]}u_{k;<}^{d_k-r}v_{k;<}^{[d_k-r]};
\end{equation}
here $u_{k;>}$ and $u_{k;<}$, $1\le k\le N$, are 
defined by
\begin{equation*}
u_{k;>}=\prod_{k\to i\in Q} x_i,\qquad  u_{k;<}=\prod_{i\to k \in Q}x_i,
\end{equation*}
where the products are taken over all edges between $k$ and mutable vertices,
and {\em stable $\tau$-monomials\/}
$v_{k;>}^{[r]}$ and $v_{k;<}^{[r]}$, $1\le k\le N$, $0\le r\le d_k$, defined by
\begin{equation}\label{stable}
v_{k;>}^{[r]}=\prod_{N+1\le i\le N+M}x_i^{\lfloor rb_{ki}/d_k\rfloor},\qquad
v_{k;<}^{[r]}=\prod_{N+1\le i\le N+M}x_i^{\lfloor rb_{ik}/d_k\rfloor},
\end{equation}
where $b_{ki}$ is the number of edges from $k$ to $i$ and $b_{ik}$ is the number of edges from $i$ to $k$; 
here, as usual, the product over the empty set is assumed to be
equal to~$1$.  In what follows we write $v_{k;>}$ instead of $v_{k;>}^{[d_k]}$ and
$v_{k;<}$ instead of $v_{k;<}^{[d_k]}$.
The right hand side of~\eqref{exchange} is called a {\it generalized exchange polynomial}.

The standard definition of the {\it quiver mutation\/} in direction $k$ is modified as follows: if both vertices $i$ and $j$
in a path $i\to k\to j$ are mutable, then this path contributes $d_k$ edges $i\to j$ to the mutated quiver $Q'$; if one of the vertices $i$ or $j$ is frozen then the path contributes $d_j$ or $d_i$ edges $i\to j$ to $Q'$. The multiplicities at the vertices do not change. Note that isolated vertices remain isolated in $Q'$.

The {\em exchange coefficient mutation\/} in direction $k$ is given by
\begin{equation}
\label{CoefMutation}
 p'_{ir}=\begin{cases}
          p_{i,d_i-r}, & \text{if $i=k$;}\\
           p_{ir}, &\text{otherwise.}
        \end{cases}
\end{equation}

Given an extended seed $\Sigma=(\x,Q,\P)$, we say that a seed
$\Sigma'=(\x',Q',\P')$ is {\em adjacent\/} to $\Sigma$ (in direction
$k$) if $\x'$, $Q'$ and $\P'$ are as above. 
Two such seeds are {\em mutation equivalent\/} if they can
be connected by a sequence of pairwise adjacent seeds. 
The set of all seeds mutation equivalent to $\Sigma$ is called the {\it generalized cluster structure\/} 
(of geometric type) in $\F$ associated with $\Sigma$ and denoted by $\GCC(\Sigma)$.

Fix a ground ring $\widehat{\AA}$ such that $\AA\subseteq\widehat\AA\subseteq\bar\AA$. The
{\it generalized upper cluster algebra\/}
$\UU(\GCC)=\UU(\GCC(\Sigma))$ is the intersection of the rings of Laurent polynomials over $\widehat{\AA}$ in cluster variables taken over all seeds in $\GCC(\Sigma)$. Let $V$ be a quasi-affine variety over $\C$, $\C(V)$ be the field of rational functions on $V$, and $\O(V)$ be the ring of regular functions on $V$. A generalized cluster structure $\GCC(\Sigma)$
in $\C(V)$ is an embedding of $\x$ into $\C(V)$ that can be extended to a field isomorphism between $\F_\C=\F\otimes\C$ and 
$\C(V)$. 
It is called {\it regular on $V$\/} if any cluster variable in any cluster belongs to $\O(V)$, and {\it complete\/} if 
$\UU(\GCC)$ tensored with $\C$ is isomorphic to $\O(V)$. The choice of the ground ring is discussed 
in~\cite[Section 2.1]{GSVDouble}.

Let $\Poi$ be a Poisson bracket on the ambient field $\F$, and $\GCC$ be a generalized cluster structure in $\F$. 
We say that the bracket and the generalized cluster structure are {\em compatible\/} if any extended
cluster $\widetilde{\x}=(x_1,\dots,x_{N+M})$ is {\em log-canonical\/} with respect to $\Poi$, that is,
$\{x_i,x_j\}=\omega_{ij} x_ix_j$,
where $\omega_{ij}\in\Z$ are constants for all $i,j$, $1\le i,j\le N+M$.

For any mutable vertex $k\in Q$ define the $y$-variable
\begin{equation}\label{yvar}
y_k=\frac{u_{k;>}^{d_k}v_{k;>}}{u_{k;<}^{d_k}v_{k;<}}.
\end{equation}
The following statement is an immediate corollary of \cite[Proposition 2.5]{GSVDouble}.

\begin{proposition}\label{compatchar}
Assume that for any mutable vertex $j\in Q$
\[
\{\log x_i,\log y_j\}=\lambda d_j\delta_{ij}\quad\text{for any $i\in Q$,}
\]
where $\lambda$ is a rational number not depending on $j$, $\delta_{ij}$ is the Kronecker symbol, 
and all Laurent monomials 
\[
\hat p_{kr}=\frac{\left(p_{kr}v_{k;>}^{[r]}v_{k;<}^{[d_k-r]}\right)^{d_k}}{v_{k;>}^r v_{k;<}^{d_k-r}} 
\]
are Casimirs of the bracket $\Poi$.
Then the bracket $\Poi$ is compatible with $\GCC(\Sigma)$.
\end{proposition}

The notion of compatibility  extends to Poisson brackets on $\F_\C$ without any changes.

 Fix an arbitrary extended cluster
$\wx=(x_1,\dots,x_{N+M})$ and define a {\it local toric action\/} of rank $s$ as a map 
$\TA^W_{\q}:\F_\C\to\F_\C$ given on the generators of $\F_\C=\C(x_1,\dots,x_{N+M})$ by the formula 
\begin{equation}
\TA^W_{\q}(\wx)=\left ( x_i \prod_{\alpha=1}^s q_\alpha^{w_{i\alpha}}\right )_{i=1}^{N+M},\qquad
\q=(q_1,\dots,q_s)\in (\C^*)^s,
\label{toricact}
\end{equation}
where $W=(w_{i\alpha})$ is an integer $(N+M)\times s$ {\it weight matrix\/} of full rank, and extended naturally to 
the whole $\F_\C$. 

Let $\wx'$ be another extended cluster in $\GCC$, then the corresponding local toric action defined by the weight matrix $W'$
is {\it compatible\/} with the local toric action \eqref{toricact} if it commutes with the sequence of (generalized) cluster transformations that takes $\wx$ to $\wx'$. If local toric actions at all clusters are compatible, they define a {\it global toric action\/} $\TA_{\q}$ on $\F_\C$ called a 
$\GCC$-extension of the local toric action \eqref{toricact}.  As shown in \cite[Section 5.2]{GSVb}, for a global toric action to be well-defined, it suffices that local toric actions at all seeds adjacent to the initial one are compatible. The following statement is equivalent to~\cite[Proposition 2.6]{GSVDouble}.

\begin{proposition}\label{globact}
The local toric action~\eqref{toricact} is uniquely $\GCC$-extendable
to a global action of $(\C^*)^s$  if  all $y$-variables $y_k$ and all Casimirs $\hat p_{ir}$
are invariant under~\eqref{toricact}.
 \end{proposition}

\section{The structure $\NGCC_n^D$} \label{doublestruct}

In this section, we provide a description of the seed $\bar\Sigma_n$ and prove that the corresponding generalized cluster structure $\NGCC_n^D$ is complete and compatible with the standard Poisson--Lie structure on $D(GL_n)$.

 First, we list some terms and notations that will be used in what follows. A notation $A_I^J$ is reserved for a submatrix of a matrix $A$ with a row set $I$ and a column set $J$. If $I$ (resp. $J$) is not specified, it is assumed that all rows (resp. columns) are selected. An interval notation $[i, j]$ is used for the index set $[i, i+1,\ldots, j]$. We call a submatrix or minor of $A$  {\em dense} if both its row and column sets are intervals. A dense minor of $A$ is called {\em trailing} if it contains the lower right entry of $A$.
 
\subsection{The initial seed} \label{init}
Let $(X,Y)\in D(GL_n) = GL_n\times GL_n$. Following \cite{GSVstaircase}, define an $N\times N$  matrix 
\begin{equation}\label{Phi}
\Phi=\Phi\left ( X, Y\right)=\left (
\begin{array}{cccc}
Y_{[2,n]} & & &  \\
X_{[2,n]} & Y_{[2,n]} & &  \\
 & \ddots &\ddots &   \\
 & & X_{[2,n]} & Y_{[2,n]}  \\
  & & & X_{[2,n]} 
 \end{array}
\right )
\end{equation}
with $N=(n-1)n$ and  put $\phhi_i=\det\Phi_{[i,N]}^{[i,N]}$ for  $1\le i\le N$. 
Further, put
$\det \left ( \lambda Y + \mu X\right ) = \sum_{i=0}^n c_i(X,Y) \mu^i \lambda^{n-i}$.

Next, we define $g_{ij}=\det X_{[i,n]}^{[j,j+n-i]}$ for $1\le j\le i\le n$, 
and, $h_{ij}=\det Y_{[i,i+n-j]}^{[j,n]}$ for $1\le i\le j\le n$; note that $\phhi_i=g_{i-N+n-1,i-N+n-1}$ for $i>N-n+1$.
The family ${\bar\FFF_n}$ of $2n^2$ functions in the ring of regular functions on $D(GL_n)$ is defined as
\[
\bar\FFF_n=\left\{\{\phhi_i\}_{i=1}^{N-n+1};\ \{g_{ij}\}_{1\le j\le i\le n};\ \{h_{ij}\}_{1\le i\le j\le n};\ 
\{\tilde c_i\}_{i=1}^{n-1}\right\}
\]
with $\tilde c_i(X,Y)=(-1)^{i(n-1)}c_i(X,Y)$ for $1\le i\le n-1$.

The corresponding quiver $\bar Q_n$ is defined below and illustrated, for the $n=4$ case, in Figure \ref{example_quiver}. 
It has $2n^2$ vertices corresponding to the functions in $\bar\FFF_n$. The $n-1$ vertices corresponding to 
$\tilde c_i(X,Y)$, $1\le i\le n-1$, are isolated; they are not shown. There are $2n$ frozen vertices corresponding to 
$g_{i1}$, $1\le i\le n$, and $h_{1j}$, $1\le j\le n$;  they are shown as squares in the figure below.
All vertices except for one are arranged into a $(2n-1) \times n$ grid; we will 
refer to vertices of the grid using their position in the grid numbered top to bottom and left to right.
The edges of $\bar Q_n$ are $(i,j) \to (i+1,j+1)$ for $i=1,\dots, 2n-2$, $j=1,\ldots, n-1$, $(i,j) \to (i,j-1)$ and 
$(i,j) \to (i-1,j)$
for $i=2,\dots, 2n-1$, $j=2,\ldots, n$, and $(i,1)\to (i-1,1)$ for $i=2,\dots,n$. Additionally, there is an oriented path
\[
(n+1,n)\to (3,1)\to (n+2,n)\to (4,1)\to\cdots\to(n,1)\to(2n-1,n). 
\]
The edges in this path are depicted as dashed in Figure \ref{example_quiver} ( the dashed style does not indicate any special features of these edges, it is for visualization purposes only). The vertex $(2,1)$ is special;
it is shown as a hexagon in the figure. The last remaining 
vertex of $\bar Q_n$ is placed to the left of the special vertex and there is an edge pointing from the former one to the latter.

\begin{figure}[ht]
\begin{center}
\includegraphics[width=7cm]{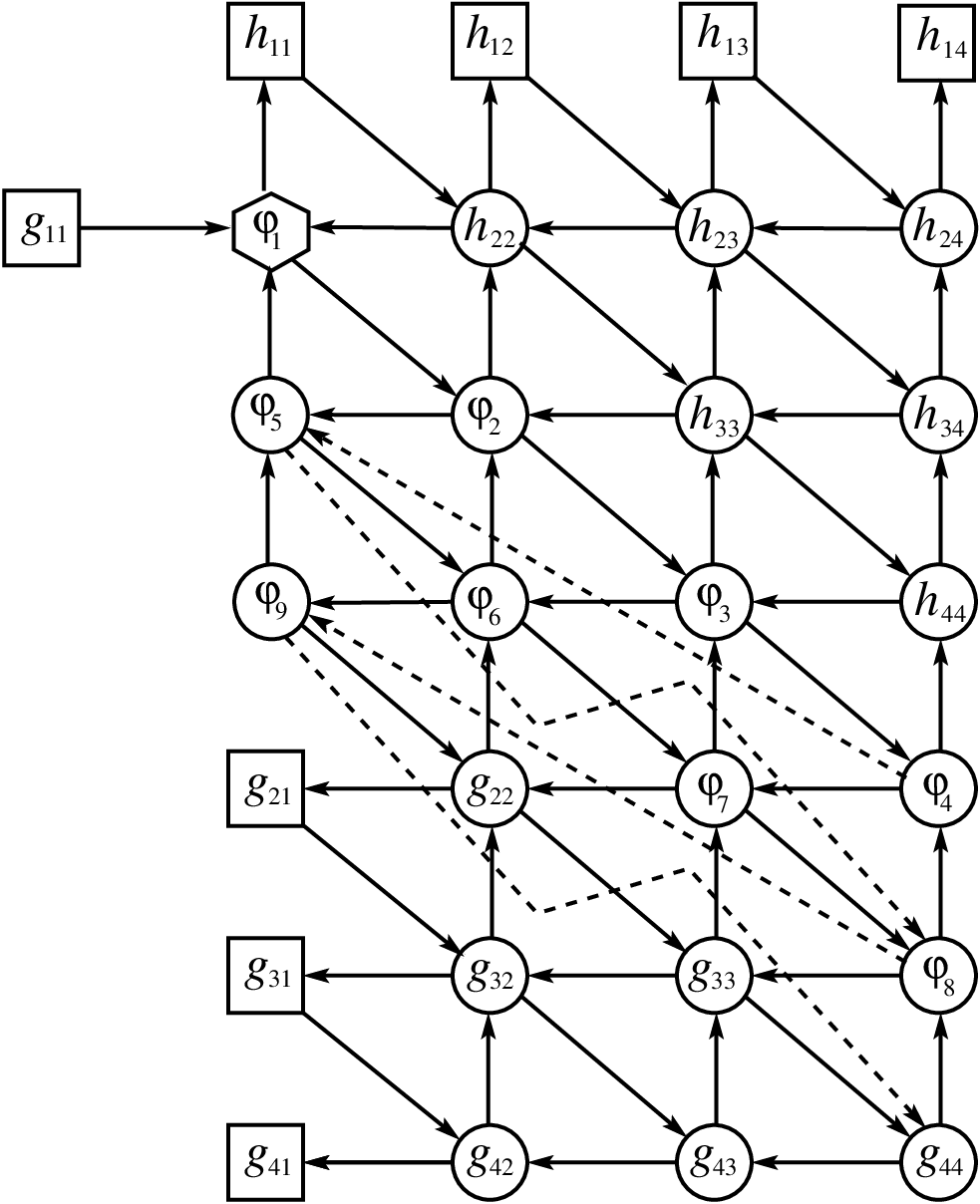}
\end{center}
\caption{Quiver $\bar Q_4$}
\label{example_quiver}
\end{figure}

Functions $h_{ij}$ are attached to the vertices $(i,j)$, $1\le i\le j\le n$, and all vertices in the upper row of $\bar Q_n$
are frozen. Functions $g_{ij}$ are attached to the vertices $(n+i-1,j)$, $1\le j\le i\le n$, $(i,j)\ne(1,1)$, and all such
vertices in the first column are frozen. The function $g_{11}$ is attached to the vertex to the left of the special one, and
this vertex is frozen. Functions $\phhi_{kn+i}$ are attached to the vertices $(i+k+1,i)$ for $1\le i\le n$, $0\le k\le n-3$;
the function $\phhi_{N-n+1}$ is attached to the vertex $(n,1)$. All these vertices are mutable. 
One can identify in $\bar Q_n$ three regions associated with three families
$\{ g_{ij}\}$, $\{h_{ij}\}$, $\{\phhi_{k}\}$. We will call vertices in these regions $g$-, $h$-, and $\fy$-vertices, respectively.
The set of strings $\bar\P_{n}$ contains a unique nontrivial string $(1,\tilde c_1(X,Y),\dots,\tilde c_{n-1}(X,Y),1)$ corresponding to the unique 
special vertex.

\begin{theorem}\label{newstructure}
The seed $\bar\Sigma_n=(\bar\FFF_n,\bar Q_n,\bar \P_n)$ defines a complete generalized cluster structure $\NGCC_n^D$ in the ring of regular functions on the Drinfeld double $D(GL_n)$ compatible with the standard Poisson--Lie structure on $D(GL_n)$.
\end{theorem}

\begin{proof} Regularity of $\NGCC_n^D$ is proved in \cite[Theorem 4.1]{GSVstaircase}. To prove compatibility, it suffices to check that the family $\bar\FFF_n$ is log-canonical with respect to the bracket $\Poi_D$, which is done in 
Section~\ref{logcan} below, and to check the compatibility conditions of Proposition \ref{compatchar}, which is done in 
Section~\ref{compat} below.
To prove the completeness, we establish a connection between cluster dynamics for standard cluster structures
on rectangular matrices and that for $\NGCC_n^D$ with certain vertices frozen, see Section \ref{complet} below. As a consequence we prove in Proposition \ref{matrixentries} that all matrix entries in $Y$, and all matrix entries in $X$ except for the 
first row are cluster variables in $\NGCC_n^D$. The entries in the first row of $X$ are treated separately in 
Lemma \ref{firstrow}.
\end{proof} 

\subsection{Log-canonicity}\label{logcan}
 Denote by $\b_\pm$ Borel  subalgebras of upper/lower triangular matrices in $\gl_n$ and by $\n_\pm$ the corresponding nilpotent ideals. Let $\pi_{>0}$, $\pi_{<0}$ be the projections of an element of $\gl_n$ onto $\n_+,\n_-$,
$\pi_0$ be the projection onto the diagonal subalgebra, 
$R_+ = \frac 1 2 \pi_0 + \pi_{>0}$. As explained in 
\cite[Section 2.2]{GSVDouble}, the {\em standard Poisson-Lie bracke}t $\Poi=\Poi_D$ on $D(GL_n)$ can be written as
\begin{equation}\label{sklyadoubleGL}
\begin{split}
\{f_1,f_2\}
=&\left\langle R_+(E_L f_1), E_L f_2\right\rangle -  \left\langle R_+(E_R f_1), E_R f_2\right\rangle\\
&+  \left\langle X \nabla_X  f_1, Y\nabla_Y f_2\right\rangle - \left\langle \nabla_X f_1 \cdot X, \nabla_Y f_2 \cdot Y\right\rangle\\
=&\left\langle R_+(E_L f_1), E_L f_2\right\rangle -  \left\langle R_+(E_R f_1), E_R f_2\right\rangle\\
&+  \left\langle E_R  f_1, Y\nabla_Y f_2\right\rangle - \left\langle E_L f_1, \nabla_Y f_2 \cdot Y\right\rangle,
\end{split}
\end{equation}
where $\nabla_X f=\left(\frac{\partial f}{\partial x_{ji}}\right)_{i,j=1}^n$ and
$\nabla_Y f=\left(\frac{\partial f}{\partial y_{ji}}\right)_{i,j=1}^n$ are the gradients of $f$ with respect to $X$ and $Y$,
respectively, the operators $E_R$ and $E_L$ are defined via
\[
E_R f = X \nabla_X f + Y\nabla_Y f,\qquad E_L f =  \nabla_X f\cdot X+ \nabla_Y f \cdot Y,
\]
and $\langle A,B\rangle=\Tr AB$ is the trace form; in what follows we
will omit the comma and write just $\langle AB\rangle$.

Note that the functions $\tilde c_i$ are Casimirs for $\Poi$.   One way to see this is by observing that any function $f$ on $D(GL_n)$ that has a property, shared by all $\tilde c_i$, that $f(A X B, A Y B)=\det A\det B\cdot f(X,Y)$ for any $X,Y,A,B\in GL_n$ is a Casimir. For such $f$, $E_R  f =  E_L  f$ is a scalar multiple of the identity matrix and so the claim follows from the second formula in \eqref{sklyadoubleGL} and the identity $\Tr AB=\Tr BA$. Thus, we only need to treat the functions in the three other subfamilies in $\bar\FFF_n$.

\begin{lemma}
\label{inv}
{\rm (i)} For any $1\le j\le i\le n$ and $1\le k \le N$,
\begin{equation}\label{inv_prop}
\begin{aligned}
&g_{ij}(X)=g_{ij}(N_+X), \qquad   h_{ji}(Y)=h_{ji}(YN_-), \\  
&\phhi_k(X,Y)= \phhi_k(N_+ X N_-, N_+ Y N_-),
\end{aligned}
\end{equation}
where  
$N_+$ is an arbitrary  unipotent upper-triangular matrix and $N_-$ is an arbitrary  unipotent lower-triangular matrix. In addition, 
$g_{ij}$ and $h_{ij}$ are homogeneous with respect to right and left multiplication of their arguments by arbitrary diagonal matrices, and $\phhi_k$ are homogeneous with respect to right and left multiplication of  $X$, $Y$ by the same pair of diagonal matrices.

{\rm (ii) } Let $g$, $h$ and $\phhi$ be three functions possessing invariance properties \eqref{inv_prop}, respectively.
Then
\begin{equation}\label{brafygh}
\begin{aligned}
\{\phhi,g\}&=\frac12\langle E_L\phhi,\nabla_X g\cdot X\rangle_0-\frac12\langle E_R\phhi, X\nabla_X g\rangle_0,\\
\{\phhi,h\}&=\frac12\langle E_R\phhi, Y\nabla_Y h\rangle_0-\frac12\langle E_L\phhi,\nabla_Y h\cdot Y\rangle_0,\\
\{g,h\}&=\frac12\langle X\nabla_X g, Y\nabla_Y h \rangle_0-\frac12\langle \nabla_X g\cdot X,\nabla_Y h\cdot Y
\rangle_0,
\end{aligned}
\end{equation}
where $\langle A,B\rangle_0=\langle AB\rangle_0=\langle\pi_0(A)\pi_0(B)\rangle$.
\end{lemma}

\begin{proof} 
(i) Invariance properties of functions $g_{ij}$, $h_{ij}$, $\phhi_k$ follow easily from their definition.

(ii) Infinitesimally, equations \eqref{inv_prop} imply that
$ X\nabla_X g, E_R\fy\in \b_+$ and  $ \nabla_Y h\cdot Y, E_L \fy\in\b_-$. 
Taking into account that $R_+(\xi)=\frac{1}{2}\pi_0(\xi)$ for any $\xi\in\b_\pm$ and that $\b_\pm \perp \n_\pm$ with respect to $\langle\cdot ,\cdot \rangle$, the result follows from \eqref{sklyadoubleGL}.
\end{proof}

The fact that any two of the three functions $g_{ij}$, $h_{pq}$, $\phhi_k$ are log-canonical follows immediately from the above
lemma.  Indeed,  the infinitesimal version of homogeneity properties described in Lemma \ref{inv}(i) states that for functions  $g=\log g_{ij}$, $h=\log h_{pq}$, $\phhi=\log \phhi_k$, projections of $\nabla_X g\cdot X$, $X\nabla_X g$,  
$Y\nabla_Y h$,  $\nabla_Y h\cdot Y$, $E_L\phhi$, $E_R\phhi$ onto the Cartan subalgebra are constant diagonal matrices and the claim then follows from \eqref{brafygh}. Log-canonicity of the families 
$\{ g_{ij}\}$ and $\{ h_{ij}\}$ is established in Lemma 5.4 in \cite{GSVDouble}.

It remains to show that  $\{\log\phhi_k, \log\phhi_l\}$ is constant for any $k$, $l$.  In fact, it suffices to consider
the case $1\le k\le l\le N-n+1$, since for $k>N-n+1$ the function $\phhi_k$ belongs to the family $\{g_{ij}\}$.

Denote $\nabla_X\log\phhi_k$, $\nabla_Y\log\phhi_k$,  $E_R \log\phhi_k$, $E_L\log\phhi_k$ by $\nabla_X^k$, 
$\nabla_Y^k$, $E_R^k$, $E_L^k$, respectively. 
As mentioned in the proof of Lemma~\ref{inv} (ii), $E_R^k \in \b_+$, $E_L^k  \in \b_-$. Consequently,  
\eqref{sklyadoubleGL} gives
\begin{equation}
\label{bra_k_l}
\begin{split}
\{\log\phhi_k, \log\phhi_l\} = &\frac 1 2 \left\langle E_L^k E_L^l\right\rangle_0 -  
\frac 1 2 \left\langle E_R^k E_R^l\right\rangle_0\\
&+  \left\langle X \nabla_X^k  Y\nabla_Y^l \right\rangle - \left\langle \nabla_X^k X \nabla_Y^l Y\right\rangle.
\end{split}
\end{equation}
It follows from  the homogeneity of functions $\fy_k$ (Lemma~\ref{inv}(i)) that  the first two terms above are constant. 
Thus, we need to evaluate 
\begin{equation}
\label{badpiece}
\left \langle X \nabla^k_X Y \nabla^l_Y \right \rangle - \left \langle  \nabla^k_X X \nabla^l_Y Y \right \rangle.
\end{equation}

For a smooth function $F$ on $\Mat_N$, we write its gradient $\nabla_\Phi F$ in a block form as
$\nabla_\Phi F = \left ( \nabla_{pq}\right )_{p=1, q=1}^{n-1, n}$, where the blocks $\nabla_{pq}$ 
are of size $n\times (n-1)$. 
For functions $\phhi_k$  viewed as functions on $\Mat_N$, we denote $\nabla_\Phi \log\phhi_k$ by 
$\nabla^k_\Phi =\left ( \nabla^k_{pq}\right )_{p=1, q=1}^{n-1, n}$. 

 Observe that if one views $l\times l$ trailing principal minor of a square matrix $A$ as a function $f(A)$ of $A$ then
\[
A \nabla f(A) = \left (\begin{array}{cc} 0 & \star \\0 & \one_{l}\end{array} \right ),\qquad 
\nabla f(A) A = \left (\begin{array}{cc} 0 & 0 \\ \star& \one_{l}\end{array} \right ).
\]
Thus
\begin{equation}
\label{nablaPhi}
\Phi \nabla_\Phi^k = \left (\begin{array}{cc} 0 & \star \\0 & \one_{N-k+1}\end{array} \right ),\qquad 
\nabla_\Phi^k \Phi = \left (\begin{array}{cc} 0 & 0 \\ \star& \one_{N-k+1}\end{array} \right ).
\end{equation}

Denote $\X = X_{[2,n]}$, $\Y = Y_{[2,n]}$.  Then \eqref{nablaPhi} translates into
\begin{equation}\label{nablaPhiEntries}
\begin{aligned}
&\X  \nabla^k_{pq} + \Y  \nabla^k_{p+1,q}=0,\quad p \geq q, \\ 
&\X  \nabla^k_{n-1, q} = 0, \quad q < n, \\
&\nabla^k_{pq} \Y + \nabla^k_{p,q+1} \X=0,\quad  p < q.
\end{aligned}
\end{equation}

Clearly, 
\begin{align*}
&\left \langle X \nabla^k_X Y \nabla^l_Y \right \rangle = \left \langle \X \sum_{p=1}^{n-1}\nabla^k_{p,p+1} \Y \sum_{q=1}^{n-1}\nabla^l_{qq} \right \rangle ,\\
&\left \langle  \nabla^k_X X \nabla^l_Y Y \right \rangle = \left \langle  \sum_{p=1}^{n-1}\nabla^k_{p,p+1} \X  \sum_{q=1}^{n-1}\nabla^l_{qq} \Y \right \rangle.
\end{align*}

Denote 
$A_{pq} = \left\langle \X \nabla^k_{p,p+1} \Y \nabla^l_{qq} \right\rangle$, 
$B_{pq} = \left\langle  \nabla^k_{p,p+1}\X \nabla^l_{qq} \Y \right\rangle$.
Using \eqref{nablaPhiEntries}, we obtain
\begin{equation*}
\begin{split}
A_{pq} = -\left \langle \X \nabla^k_{p,p+2} \X \nabla^l_{qq} \right \rangle = 
\left \langle \X \nabla^k_{p,p+2} \Y \nabla^l_{q+1,q} \right \rangle = \cdots \\
= -\left \langle \X \nabla^k_{p,p+s+1} \X \nabla^l_{q+s-1,q} \right \rangle
 = \left \langle \X \nabla^k_{p,p+s+1} \Y \nabla^l_{q+s,q} \right \rangle = \cdots\\ =
\begin{cases} 
-\left \langle \X \nabla^k_{p,p-q +n+1} \X \nabla^l_{n-1,q} \right \rangle=0, & p<q,\\
 \left \langle \X \nabla^k_{pn} \Y \nabla^l_{q-p+n-1,q} \right \rangle, & p\geq q.
\end{cases}
\end{split}
\end{equation*}

Similarly,
\[
B_{pq} = \begin{cases} 
0, & p\leq q,\\
- \left \langle\nabla^k_{pn} \Y \nabla^l_{q-p+n,q} \Y\right \rangle, & p > q.
\end{cases} 
\]
Therefore,
\begin{equation}\label{sums}
\begin{aligned}
&\left \langle X \nabla^k_X Y \nabla^l_Y \right \rangle - \left \langle  \nabla^k_X X \nabla^l_Y Y \right \rangle \\
&= \sum_{1\leq q\leq p\leq n-1} \left \langle \X \nabla^k_{pn} \Y \nabla^l_{q-p+n-1,q} \right \rangle \ 
+\sum_{1\leq q <  p\leq n-1} \left \langle \Y \nabla^k_{pn} \Y \nabla^l_{q-p+n,q} \right \rangle \\
&= \sum_{1\leq q\leq p\leq n-1} \left \langle \X \nabla^k_{pn} \Y \nabla^l_{q-p+n-1,q} \right \rangle \ 
+\sum_{1\leq q \leq  p\leq n-2} \left \langle \Y \nabla^k_{p+1,n} \Y \nabla^l_{q-p+n-1,q} \right \rangle \\
&= \sum_{1\leq q\leq  n-1} \left \langle \left ( \Phi \nabla_\Phi^k\right )_{nn} \Y \nabla^l_{qq} \right \rangle \ 
+\sum_{1\leq q \leq  p\leq n-2} \left \langle \left ( \Phi \nabla_\Phi^k\right )_{p+1,n} \Y \nabla^l_{q-p+n-1,q} 
\right \rangle .
\end{aligned}
\end{equation}

Since  $\left ( \Phi \nabla_\Phi^k\right )_{nn} =\one_{n-1}$ for $k\le N-n+1$, the first sum above is equal to 
$\sum_{1\leq q\leq  n-1} \left \langle \Y \nabla^l_{qq} \right \rangle$. Furthermore, 
$\nabla^l_{q-p+n-1,q}=0$ for $l>q(n-1)+1$, while
$\left ( \Phi \nabla_\Phi^k\right )_{p+1,n}=0$ for $k\le p(n-1)+1$. Thus, the summand in the second sum in~\eqref{sums} 
is zero unless
$p(n-1)+1 < k\le l \le q(n-1)+1$, which is impossible since $q\leq p$. We conclude that
\[
\left \langle X \nabla^k_X Y \nabla^l_Y \right \rangle - \left \langle  \nabla^k_X X \nabla^l_Y Y \right \rangle = \sum_{1\leq q\leq  n-1} \left \langle \Y \nabla^l_{qq} \right \rangle
\]
for $1\le k\le l\le N-n+1$.

Let us show that  the right hand side above is constant. To this end, first observe that 
\[
\sum_{1\leq q\leq  n-1} \left \langle \Y \nabla^l_{qq} \right \rangle = \left.\frac {d} {d t}\right|_{t=0} \log\phhi_l 
\left (X, e^t Y \right ). 
\]
Furthermore, 
\[
\Phi \left (X, e^t Y \right )= \diag\left ( e^{nt} \one_{n-1},\ldots, e^{t} \one_{n-1}\right ) \Phi \left (X, Y \right )\diag\left ( e^{-(n-1)t} \one_{n},\ldots, e^{-t} \one_{n}\right ) .
\] 
Then $\phhi_l\left (X, e^t Y \right )=\det \Phi\left (X, e^t Y \right )_{[l,N]}^{[l,N]} = e^{\varkappa_l t}
\phhi_l\left (X, Y \right )$, where constant coefficients $\varkappa_l$ can be arranged into a matrix
\begin{equation}
\label{kappas}
\left (\varkappa_{\mu (n-1) + \sigma}\right )_{\mu =0\ \sigma=1}^{n-1 \  n-1} = 
\left ( \begin{array} {ccccc}
{n \choose 2} & {n \choose 2}-1 & {n \choose 2}-2 & \cdots & {n-1\choose 2}+1\\
{n-1\choose 2} & {n-1\choose 2} & {n-1\choose 2}-1 & \cdots & {n-2\choose 2}+1\\
{n-2\choose 2} & {n-2\choose 2} & {n-2\choose 2} & \cdots & {n-3\choose 2}+1\\
\vdots & \vdots & \vdots & \ddots & \vdots\\
1 & 1 & 1 & \cdots & 1\\
0 & 0 & 0 & \cdots & 0
\end{array}
\right ).
\end{equation}
Therefore, $\sum_{1\leq q\leq  n-1} \left \langle \Y \nabla^l_{qq} \right \rangle = \varkappa_l$ and we are done.

\subsection{Compatibility} \label{compat}
To prove the compatibility statement, we start with the following lemma, which is a direct analog of \cite[Theorem 6.1]{GSVDouble}.

\begin{lemma}\label{gtaind}
The action 
\begin{equation}\label{lraction}
(X,Y)\mapsto (T_1 X T_2, T_1 Y T_2)
\end{equation}
 of right and left multiplication by diagonal matrices 
is $\NGCC_n^D$-extendable to a global toric action on $\FF_{\C}$.
\end{lemma}

\begin{proof} As explained in \cite{GSVDouble}, Proposition \ref{globact} implies that it suffices to check that 
$y$-variables \eqref{yvar} are homogeneous functions of degree zero with respect to the action \eqref{lraction} (note
that the functions $\hat p_{1r}$ for $\NGCC_n^D$ are the same as for $\GCC_n^D$). To this effect, 
we define weights $\xi_L(f)=\pi_0(E_L\log f)$ and $\xi_R(f)=\pi_0(E_R\log f)$. For $1\leq i \leq j \leq n$, let $\Delta(i,j)$ denote a diagonal matrix with $1$'s in the entries $(i,i), \ldots, (j,j)$ and $0$'s everywhere else.
Direct computation shows that
\begin{equation}\label{weights}
\begin{aligned}
& \xi_L(g_{ij})= \Delta(j, n+j -i),\quad\xi_R(g_{ij})= \Delta(i,n),  \\
& \xi_L(h_{ij})= \Delta(j, n),\quad   \xi_R(h_{ij})= \Delta(i,n+i-j),\\
& \xi_L(\phhi_k)=(n-2-\lambda_k)\Delta(1,n)+\Delta(\rho_k,n),\\
& \xi_R(\phhi_k)=(n-1-\mu_k)\Delta(2,n)+\Delta(\sigma_k+1,n),
\end{aligned}
\end{equation}
where  $\lambda_k$, $\rho_k$, $\mu_k$, and $\sigma_k$ are defined via $k=\lambda_kn+\rho_k$, $1\le\rho_k\le n$, and $k=\mu_k(n-1)+\sigma_k$, $1\le\sigma_k\le n-1$. 
Now the verification of the claim above becomes straightforward. It is  
based on the description of $\bar Q_n$ in Section~\ref{init} 
and the fact that for a Laurent monomial in 
homogeneous functions $M=\psi_1^{\alpha_1}\psi_2^{\alpha_2}\cdots$ the 
right and left weights are  $\xi_{R,L}(M) = \alpha_1 \xi_{R,L}(\psi_1) + \alpha_2 \xi_{R,L}(\psi_2) + \cdots$. 

For example, let $v$ be the vertex associated with the function $\phhi_k$, $n+1\le k\le N-n$. To treat the left weight
of $y_v$ we have to consider 
the following three cases.

{\it Case 1:} $\lambda_{k-1}=\lambda_k=\lambda_{k+1}=\lambda$. Consequently, $\lambda_{k-n}=\lambda_{k-n+1}=\lambda-1$, 
$\lambda_{k+n-1}=\lambda_{k+n}=\lambda+1$ and $\rho_{k-1}=\rho_{k+n-1}=\rho$, $\rho_{k-n}=\rho_{k+n}=\rho+1$, $\rho_{k+1}=\rho_{k-n+1}=
\rho+2$. Therefore, \eqref{weights} yields
\begin{align*}
\xi_L&(y_v)\\
&=\xi_L(\phhi_{k-1})+\xi_L(\phhi_{k+n})+\xi_L(\phhi_{k-n+1})-\xi_L(\phhi_{k+1})-\xi_L(\phhi_{k-n})-\xi_L(\phhi_{k+n-1})\\
&=(n-2-\lambda)\Delta(1,n)+\Delta(\rho,n)+(n-2-\lambda-1)\Delta(1,n)+\Delta(\rho+1,n)\\
&+(n-2-\lambda+1)\Delta(1,n)+\Delta(\rho+2,n)-(n-2-\lambda)\Delta(1,n)-\Delta(\rho+2,n)\\
&-(n-2-\lambda+1)\Delta(1,n)-\Delta(\rho+1,n)-(n-2-\lambda-1)\Delta(1,n)-\Delta(\rho,n)=0.
\end{align*}

{\it Case 2:} $\lambda_{k-1}=\lambda_k=\lambda$, $\lambda_{k+1}=\lambda+1$. Consequently, $\lambda_{k-n}=\lambda-1$, $\lambda_{k-n+1}=\lambda$, 
$\lambda_{k+n-1}=\lambda_{k+n}=\lambda+1$ and $\rho_{k-1}=\rho_{k+n-1}=n-1$, $\rho_{k-n}=\rho_{k+n}=n$, $\rho_{k+1}=\rho_{k-n+1}=1$. Therefore, \eqref{weights} yields
\begin{align*}
\xi_L&(y_v)\\
&=\xi_L(\phhi_{k-1})+\xi_L(\phhi_{k+n})+\xi_L(\phhi_{k-n+1})-\xi_L(\phhi_{k+1})-\xi_L(\phhi_{k-n})-\xi_L(\phhi_{k+n-1})\\
&=(n-2-\lambda)\Delta(1,n)+\Delta(n-1,n)+(n-2-\lambda-1)\Delta(1,n)+\Delta(n,n)\\
&+(n-2-\lambda)\Delta(1,n)+\Delta(1,n)-(n-2-\lambda-1)\Delta(1,n)-\Delta(1,n)\\
&-(n-2-\lambda+1)\Delta(1,n)-\Delta(n,n)-(n-2-\lambda-1)\Delta(1,n)-\Delta(n-1,n)=0.
\end{align*}

{\it Case 3:} $\lambda_{k-1}=\lambda$, $\lambda_k=\lambda_{k+1}=\lambda+1$. Consequently, $\lambda_{k-n}=\lambda_{k-n+1}=\lambda$,  
$\lambda_{k+n-1}=\lambda+1$, $\lambda_{k+n}=\lambda+2$ and $\rho_{k-1}=\rho_{k+n-1}=n$, $\rho_{k-n}=\rho_{k+n}=1$, 
$\rho_{k+1}=\rho_{k-n+1}=2$. Therefore, \eqref{weights} yields
\begin{align*}
\xi_L&(y_v)\\
&=\xi_L(\phhi_{k-1})+\xi_L(\phhi_{k+n})+\xi_L(\phhi_{k-n+1})-\xi_L(\phhi_{k+1})-\xi_L(\phhi_{k-n})-\xi_L(\phhi_{k+n-1})\\
&=(n-2-\lambda)\Delta(1,n)+\Delta(n,n)+(n-2-\lambda-2)\Delta(1,n)+\Delta(1,n)\\
&+(n-2-\lambda)\Delta(1,n)+\Delta(2,n)-(n-2-\lambda-1)\Delta(1,n)-\Delta(2,n)\\
&-(n-2-\lambda)\Delta(1,n)-\Delta(1,n)-(n-2-\lambda-1)\Delta(1,n)-\Delta(n,n)=0.
\end{align*}

The right weight of $y_v$ is treated in a similar way.
\end{proof}

The verification of compatibility conditions in Proposition \ref{compatchar} is based on relations \eqref{brafygh}, \eqref{bra_k_l} and 
\eqref{weights}. The case when the vertices $i$ and $j$ belong to different regions in $\bar Q_n$, or are simultaneously non-diagonal 
$g$- or $h$-vertices, is treated in the same way as in \cite{GSVDouble}. For the case when both $i$ and $j$ are $\phhi$-vertices, the check is based on the following claim  in which we assume $\varkappa_{N+1}=0$.

\begin{proposition}\label{lambdamukappa}
{\rm (i)} For any $k$, $1\le k\le N$, either $\lambda_k=\mu_k$ and $\varkappa_k=\varkappa_{k+1}+1$ or $\lambda_k=\mu_k-1$ and
$\varkappa_k=\varkappa_{k+1}$.

 {\rm (ii)} For any $k$, $1\le k\le N-n$,
\begin{equation}\label{almostperiodic}
\varkappa_k-\varkappa_{k+1}=\begin{cases} \varkappa_{k+n}-\varkappa_{k+n+1}, \qquad \sigma_k\ne n-1,\\
                                          \varkappa_{k+n}-\varkappa_{k+n+1}+1, \qquad \sigma_k=n-1.
														\end{cases}
\end{equation}
\end{proposition}

\begin{proof} Follows immediately from \eqref{kappas} and the definitions of $\lambda_k$, $\rho_k$, $\mu_k$, 
and $\sigma_k$.
\end{proof}

Let us check the compatibility condition for the case when
 $v$ is the vertex associated with the function $\phhi_i$ and 
$u$ is the vertex associated with the function $\phhi_j$, $n+1\le j\le N-n$. There are five possible cases: (i) $i\le j-n$; 
(ii) $j-n+1\le i\le j-1$; (iii) $i=j$; (iv) $j+1\le i\le j+n-1$; (v) $j+n\le i$. 

In the first case, the left hand side of the condition in Proposition \ref{compatchar} is computed
directly via \eqref{bra_k_l}. The first term equals 
\begin{align*}
\frac12\langle\xi_L(\phhi_i)\xi_L(y_u)\rangle&=\frac12\langle\xi_L(\phhi_i)\xi_L\left({\phhi_{j-1}}/{\phhi_{j}}\right)\rangle-
\frac12\langle\xi_L(\phhi_i)\xi_L\left({\phhi_{j+n-1}}/{\phhi_{j+n}}\right)\rangle\\
&+
\frac12\langle\xi_L(\phhi_i)\xi_L\left({\phhi_{j}}/{\phhi_{j+1}}\right)\rangle-
\frac12\langle\xi_L(\phhi_i)\xi_L\left({\phhi_{j-n}}/{\phhi_{j-n+1}}\right)\rangle
\end{align*}
and vanishes by Lemma \ref{gtaind}. 
Similarly, the second term equals 
\begin{align*}
-\frac12\langle\xi_R(\phhi_i)\xi_R(y_u)\rangle&=-\frac12\langle\xi_R(\phhi_i)\xi_R\left({\phhi_{j-1}}/{\phhi_{j}}\right)\rangle+
\frac12\langle\xi_R(\phhi_i)\xi_R\left({\phhi_{j+n-1}}/{\phhi_{j+n}}\right)\rangle\\
&-
\frac12\langle\xi_R(\phhi_i)\xi_R\left({\phhi_{j}}/{\phhi_{j+1}}\right)\rangle+
\frac12\langle\xi_R(\phhi_i)\xi_R\left({\phhi_{j-n}}/{\phhi_{j-n+1}}\right)\rangle
\end{align*}
and vanishes by Lemma \ref{gtaind}. 
The third term equals 
\[
(\varkappa_{j-1}-\varkappa_j)-(\varkappa_{j+n-1}-\varkappa_{j+n})+(\varkappa_j-\varkappa_{j+1})-(\varkappa_{j-n}-\varkappa_{j-n+1})
\]
 and vanishes  by \eqref{almostperiodic} since $\sigma_{j-1}=\sigma_{j-n}$.

In the second case, in order to use \eqref{bra_k_l}, one has to swap the  arguments in the brackets $\{\log\phhi_i,\log\phhi_{j-n}\}_D$ and  
$\{\log\phhi_i,\log\phhi_{j-n+1}\}_D$. Consequently, the first term contributes 
\[
\langle\xi_L(\phhi_i)  \xi_L(\phhi_{j-n}/\phhi_{j-n+1})\rangle,
\] 
the second term contributes 
\[
-\langle\xi_R(\phhi_i)  \xi_R(\phhi_{j-n}/\phhi_{j-n+1})\rangle,
\]
 and the third term contributes $\varkappa_{j-n}-\varkappa_{j-n+1}$. 
By Proposition \ref{lambdamukappa}(i), either $\varkappa_{j-n}-\varkappa_{j-n+1}=0$ and  
$\lambda_{j-n}=\mu_{j-n}-1$ which implies
\[
\langle\xi_L(\phhi_i)  \xi_L(\phhi_{j-n}/\phhi_{j-n+1})\rangle=\langle\xi_R(\phhi_i) 
\xi_R(\phhi_{j-n}/\phhi_{j-n+1})\rangle=n-\lambda_{j-n}-2,
\]
 or
$\varkappa_{j-n}-\varkappa_{j-n+1}=1$ and  $\lambda_{j-n}=\mu_{j-n}$ which implies 
\begin{align*}
\langle\xi_L(\phhi_i)  \xi_L(\phhi_{j-n}/\phhi_{j-n+1})\rangle&=n-\lambda_{j-n}-2, \\
\langle\xi_R(\phhi_i) \xi_R(\phhi_{j-n}/\phhi_{j-n+1})\rangle&=n-\lambda_{j-n}-1.
\end{align*}
In both cases the total additional contribution vanishes.

Cases (iii)--(v) are treated in a similar way.

\subsection{Completeness} \label{complet}
We start with the following proposition.

\begin{proposition}
\label{tallmatrix}
There exists an $(n-1)\times(n-1)$ unipotent upper triangular matrix $G=G(X, Y)$ such that

{\rm(i)} entries of $G$ are rational functions in $X, Y$ whose denominators are monomials in cluster variables 
$\phhi_{k n +1}$, $k=1,\ldots, n-2$, and 

{\rm(ii)} the $(2n-1)\times n$ matrix $S=\begin{bmatrix} Y\\ G X_{[2,n]}\end{bmatrix}$ satisfies
\begin{equation}
\label{tallmatrixminors}
 \det S_{[ n-i+1 +k, n+k]}^{[n-i+1 , n]} = \frac{\phhi_{k n -i +1 }} {\phhi_{k n +1}}, 
\qquad k=1,\ldots, n-2,\quad i=1,\ldots, n.
\end{equation}
\end{proposition}

\begin{proof} We will establish the claim by applying  to the matrix $\Phi$ a sequence of $n-2$ transformations that do not affect 
its trailing principal minors. After the $k$th transformation, every $X$-block of $\Phi$ will be multiplied on the left by the same 
unipotent upper triangular matrix $G_k$, while the $Y$-block in the $j$th block row of $\Phi$ will be replaced by 
$(G_k)_{[j,n-1]}^{[j,n-1]}Y_{[j+1,n]}$ preceded with $j-1$ zero rows for $j=2,\dots, k$ and by $Y_{[k+1,n]}$ 
preceded with $k-1$ zero rows for $j=k+1,\dots,n-1$.

On the first step, we use the submatrix $\Psi_{1} = \Phi_{[n+1, N]}^{[n+1, N]}$ to eliminate all the entries in the
row immediately above $\Psi_1$  in the submatrix $\Phi^{[n+1, N]}$. To this effect, we multiply the block rows from $2$ to $n$ of $\Phi$ on the left by a 
block T\"oplitz upper triangular matrix
\[
\begin{bmatrix} G_{11} & G_{12} &\dots & G_{1,n-1} \\ 0 & G_{11} &\dots & G_{1,n-2}\\
\vdots & \vdots& \ddots& \vdots\\0 & 0 & \dots & G_{11} \end{bmatrix}, 
\]
where $G_{1j}$ for $j>1$ are $(n-1)\times(n-1)$ matrices with the only non-zero entries lying in the first row 
and $G_{11}$ is  a unipotent upper triangular matrix with the only off-diagonal non-zero entries lying in the first row. 
Note that the denominator of the non-trivial entries in $G_{1j}$ is equal to $\det \Psi_{1} = \phhi_{n +1}$. 
As a result, all $X$-blocks of $\Phi$ are multiplied by $G_1=G_{11}$, and $Y$-blocks in the block rows $2,\dots,n-1$ 
are replaced by $Y_{[3,n]}$ preceded with the zero row. All zero blocks are not changed. The obtained matrix is block lower triangular, and hence \eqref{tallmatrixminors} is valid for $k=1$, $i=1,\dots,n$, for the matrix 
$S_1=\begin{bmatrix} Y\\ G_{1} X_{[2,n]}\end{bmatrix}$.

On the second step, we use the submatrix $\Psi_{2} = \Phi_{[2n+1, N]}^{[2n+1, N]}$  to eliminate all the entries in the
row immediately above $\Psi_2$ in the submatrix formed by columns $[2n+1, N]$ of the matrix obtained on the
previous step . To this effect, we multiply the block rows from $2$ to $n$ of this matrix on the left by a block
upper triangular matrix
\[
\begin{bmatrix} G_{21} & 0 & 0 & \dots & 0\\
0 & G_{21} & G_{22} &\dots & G_{2,n-2} \\
0 & 0 & G_{21} &\dots & G_{2,n-3} \\
\vdots & \vdots & \vdots& \ddots& \vdots\\0 & 0 & 0 & \dots & G_{21} \end{bmatrix}, 
\]
where $G_{2j}$ for $j>1$ are $(n-1)\times(n-1)$ matrices with the only non-zero entries lying in the second row 
and $G_{21}$ is a unipotent upper triangular with the same property. 
Note that the denominator of the non-trivial entries in $G_{2j}$ is equal to $\det \Psi_{2} = \phhi_{2n +1}$. 
As a result, all $X$-blocks  are multiplied by $G_2=G_{21}G_{11}$, the $Y$-block in the second block row is replaced
by $(G_2)_{[2,n-1]}^{[2,n-1]}Y_{[3,n]}$ preceded by the zero row, and $Y$-blocks in the block rows $3,\dots,n-1$ are replaced by $Y_{[4,n]}$ preceded by two zero rows. The submatrix of the obtained matrix lying in rows and columns $n+1,\dots,N$ is 
block lower triangular, and hence \eqref{tallmatrixminors} is valid for $k=2$, $i=1,\dots,n$, for the matrix 
$S_2=\begin{bmatrix} Y\\ G_{2} X_{[2,n]}\end{bmatrix}$. It is also valid for $k=1$, $i=1,\dots,n$, since the corresponding minors coincide with those for the matrix $S_1$.

Continuing in the same fashion, we define $G_{31}, \ldots, G_{n-2,1}$ such that the product $G= G_{n-2,1} \cdots G_{11}$ satisfies 
properties (i) and (ii) of the claim. Note that the $k$th row of $G$ coincides with the $k$th row of $G_{k1}$.
\end{proof}

The following proposition can be proved in a similar way.

\begin{proposition}
\label{longmatrix}
There exists an $n\times n$ unipotent lower triangular matrix $H=H(X, Y)$ such that

{\rm (i)} entries of $H$ are rational functions in $X$, $Y$ whose denominators are monomials in cluster variables 
$\phhi_{k (n-1) +1}$, $k=1,\ldots, n-1$, and 

{\rm (ii)} the $(n-1)\times 2 n$ matrix $T=\left [X_{[2,n]} \ Y_{[2,n]} H \right ]$ satisfies
\begin{equation}
\label{longmatrixminors}
 \det T_{[ n-i, n-1]}^{[n+k-i+1 , n+k]} = \frac{\phhi_{(n-k) (n -1) - i +1 }} {\phhi_{(n-k) (n-1) +1}}, \qquad k=1,\ldots, n-1, 
\quad i=1,\ldots, n-1.
\end{equation}
\end{proposition}

We will use Propositions \ref{tallmatrix}, \ref{longmatrix} to establish the following

\begin{proposition}
\label{matrixentries}
All matrix entries of $X_{[2,n]}$ and $Y$ are cluster variables in $\NGCC_n^D$.
\end{proposition}

\begin{proof}  We will use the comparison with the standard  cluster structure $\widetilde\CC_{l\times m}$ on  $\Mat_{l\times m}$ which is isomorphic to the cluster structure on the big cell in the Grassmannian $\Gr(l, l+m)$ described in \cite[Section 4.2.2]{GSVb} via the identification of a matrix $A\in \Mat_{l\times m}$ with a representative $\left [\one_l\  (A W_m) \right ]$ of an element in  $\Gr(l, l+m)$, where $W_m$ is the $m\times m$ antidiagonal matrix.
First, we  establish the claim for matrix entries of $Y=(y_{ij})_{i,j=1}^n$. Let us temporarily treat the vertices in $\bar Q_n$ that correspond to variables
$\phhi_{k n +1}$, $k=0,\ldots, n-2$, as frozen. Then the vertex corresponding to $g_{11}$ becomes isolated. Denote the quiver formed by the rest of the non-isolated vertices by $\widehat Q_n$, and let $\widehat \FFF_n\subset\bar \FFF_n$ be the corresponding subset of cluster variables. 

Define a new collection of variables $\tilde\FFF_n=\{\tilde f\: f\in\widehat\FFF_n\}$ via $\tilde\phhi_{k n -i +1 } = 
\frac{\phhi_{k n -i +1 }} {\phhi_{k n +1}}$, $k=1,\ldots, n-2$, $i=1,\ldots n$, and $\tilde f=f$ for all other 
$f\in\widehat \FFF_n$. Denote by $\tilde Q_n$ the quiver obtained from $\widehat Q_n$ via deletion of all edges 
that are dashed in Fig.~\ref{example_quiver}. 

Let us now define a $(2n-1)\times n$ matrix $S$  as in Proposition \ref{tallmatrix}. Then \eqref{tallmatrixminors} ensures that the collection $\tilde\FFF_n$ consists of  all dense minors of $S$ containing entries of the last row or column of $S$, where the minor that has an $(i,j)$-entry of $S$ in a top left corner is attached to the $(i,j)$ vertex in the grid that describes $\tilde Q_n$. Viewed this way, $\widetilde\Sigma=\left (\tilde\FFF_n, \tilde Q_n\right )$ becomes the initial seed for the standard cluster structure $\widetilde\CC_{(2n-1)\times n}$ on  $\Mat_{(2n-1)\times n}$. Note that every exchange relation in this seed is obtained via
dividing the corresponding exchange relation of $\NGCC_n^D$ by an appropriate monomial in variables $\phhi_{kn+1}$, 
$k=1,\dots,n-2$, which are frozen in $\widehat\Sigma=\left (\widehat\FFF_n, \hat Q_n\right )$. Applying 
\cite[Lemma 8.4]{GSVMem} repeatedly, we conclude that if $\phhi$ is a cluster variable in $\NGCC_n^D$ obtained via an 
arbitrary sequence of mutations not involving mutations at $\phhi_{k n +1}$,  $k=0,\ldots, n-2$, then the result of the 
same sequence of mutations in $\widetilde\CC_{(2n-1)\times n}$
is $\tilde\phhi=\frac {\phhi}{M}$, where $M$ is a Laurent monomial in  $\phhi_{k n +1}$, $k=1,\ldots, n-2$.

Since all matrix entries of $S$ are cluster variables in $\widetilde\CC_{(2n-1)\times n}$, the latter observation means 
that, for any $i, j \in [1,n]$, there is a cluster variable $\phhi\in \NGCC_n^D$ obtained via a
sequence of mutations not involving mutations at $\phhi_{k n +1}$, $k=0,\ldots, n-2$, and
such that $\phhi=y_{ij} M$, where $M$ is a Laurent monomial in  $\phhi_{k n +1}$, $k=1,\ldots, n-2$. 
We will now show that $M=1$. Indeed, since $\phhi$ is a regular function in $X, Y$, and all functions 
in $\FFF_n$ are irreducible via \cite[Lemma 4.2]{GSVstaircase},
all factors in $M$ have nonnegative  degrees. On the other hand,  in terms of elements of $\bar\FFF_n$, $\phhi$ is a 
polynomial in $\phhi_{k n +1}$, $k=1,\ldots, n-2$. Furthermore, as a polynomial in each $\phhi_{k n +1}$ it has a nonzero constant term. This claim is obvious for every single mutation away from the initial cluster, and then is verified inductively. This implies that $M=1$.
Thus all matrix entries of $Y$ are cluster variables in $\NGCC_n^D$.

\begin{figure}[ht]
\begin{center}
\includegraphics[width=10cm]{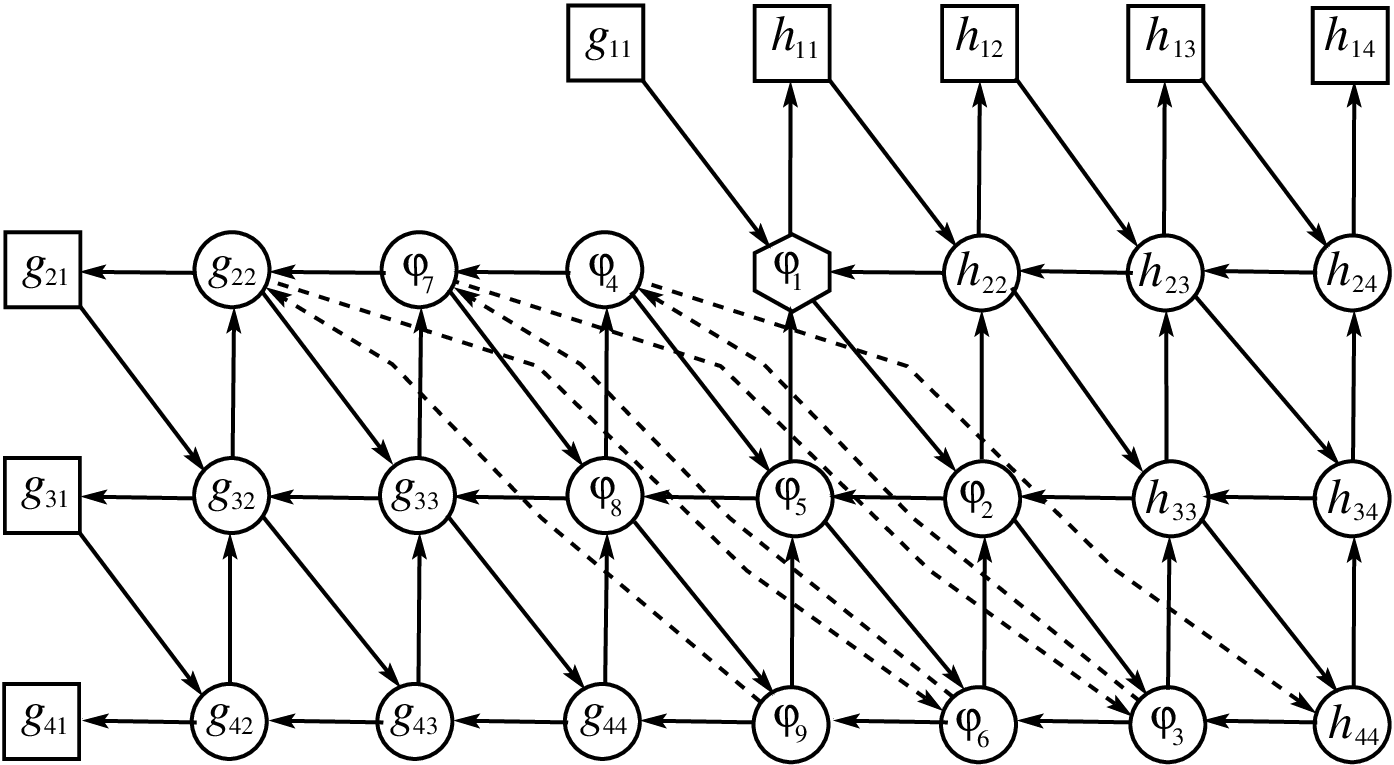}
\end{center}
\caption{Rearranged quiver $\bar Q_4$}
\label{example_quiver_alt}
\end{figure}

To treat matrix entries of $X_{[2,n]}$, we start with rearranging the vertices of $\bar Q_n$. All vertices except for those 
corresponding to $g_{11}$ and $h_{1i}$, $i=1,\dots,n$, are arranged into an $(n-1)\times 2n$ grid. It is obtained by moving
the lower $(n-1)\times n$ part and placing it on the left from the upper $(n-1)\times n$ part; the remaining $n+1$ 
vertices are placed above as an additional row aligned on the right. All former dashed edges become regular, and the 
path  
\[
(n-1,n+1)\to(1,2)\to(n-1,n+2)\to(1,3)\to\dots\to(1,n)\to(n-1,2n)
\]
becomes dashed; 
see Fig.~\ref{example_quiver_alt} for the rearranged version of $\bar Q_4$. To proceed further, we 
temporarily freeze vertices that correspond to $\phhi_{k (n-1) +1}$, $k=0,\ldots, n-2$, and $h_{2j}$, $j=2,\ldots, n$, and compare the result with an initial seed for the standard cluster structure on 
$\Mat_{(n-1)\times 2n}$ defined by the matrix $T$ given in Proposition \ref{longmatrix}.
\end{proof}

Recall that the generalized exchange relation for  $\phhi_1$ is given by
\begin{equation}
\label{longidXY}
\phhi_1 \phhi'_1=\sum_{i=0}^n c_i(X,Y) \left ( (-1)^{n-1} h_{22}\phhi_{n+1}\right )^i \phhi_2 ^{n-i},
\end{equation}
where $\phhi'_1$ is a polynomial in the entries of $X$ and $Y$, see \cite[Section~4]{GSVstaircase}. 
Denote $\bar\FFF'_n=\bar\FFF_n\setminus\{\phhi_1\}\cup\{\phhi_1'\}$.

The following lemma implies that entries of the first row of $X$ belong to the generalized upper cluster 
algebra $\UU (\NGCC_n^D)$.

\begin{lemma}\label{firstrow}
Every matrix entry $x_{1i}$ can be expressed as $\frac {P(X,Y)} {\phhi_1}$, where $P$ is a polynomial in matrix entries 
of $Y$, $X_{[2,n]}$  and functions $c_j(X,Y)$, $j=1,\ldots, n$. Alternatively, $x_{1i}$ can be expressed as 
Laurent polynomials in terms of cluster variables in 
$\bar\FFF'_n$.
\end{lemma}

\begin{proof} Denote  by $\bar X$ the matrix obtained from $X$ by setting all entries of the first row to $0$. 
For $1\le j \le n$, $c_j(X,Y)$ is a linear function in matrix entries $x_{1i}$, and so we can write
\begin{equation}
\label{TheSystem}
c_j(X,Y) = c_j(\bar X, Y)  + \sum_{i=1}^n x_{1i} z_{ij}(\bar X, Y), 
\end{equation}
where 
$z_{ij}(\bar X, Y)$ are polynomials in the entries of $X_{[2,n]}$ and $Y$. Thus we obtain a linear system for 
the entries $x_{1i}$  
and $Z= \left (z_{ij}(\bar X, Y)\right )_{i,j=1}^n$ is the matrix of this system. Clearly, solutions of the 
system are polynomials in the entries of $X_{[2,n]}$ and $Y$ divided by $\det Z$. 

Note that $z_{ij}(\bar X,Y)$ is a polynomial of 
degree $n-j-1$ in the entries of $X$ and of degree $j$ in the entries of $Y$ and so, both $\det Z$ and $\phhi_1(X,Y)$ are polynomials of total degree $\frac {n(n-1)}{2}$ in terms of both $X$ and $Y$.
We will now show that, up to a scalar multiple, $\det Z$ coincides with $\phhi_1(X,Y)$. To this end, we will demonstrate that 
$\det Z=0$ implies $\fy_1=0$.   Since $\phhi_1$ is irreducible (\cite[Lemma 4.2]{GSVstaircase}), this means that  $\det Z$ and $\phhi_1(X,Y)$ differ by a constant multiple.

To prove the implication $\det Z=0 \Rightarrow\fy_1=0$ , suppose that $\det Z$ vanishes. Then the system \eqref{TheSystem} is still solvable, but its solution is not unique. Consequently, there exists a non-zero row vector 
$v^T$ such that $\det\left ( e_1 v^T + X + \lambda Y \right ) = \det\left (X + \lambda Y \right )$. The determinant on the left
 is evaluated via the Schur complement:
\begin{equation}
\label{Schur}
\det\begin{bmatrix} X+\lambda Y & e_1\\
                       v^T & -1
	\end{bmatrix}
	=\det (X+\lambda Y) \left(1+v^T(X+\lambda Y)^{-1}e_1\right),										
\end{equation}
which means that $v^T(X + \lambda Y )^{-1}e_1 = 0$ for any $\lambda$.  Equivalently, $v^T Y^{-1} U^j e_1=0$ 
for $j=0, 1, 2,\ldots$, where $U = X Y^{-1}$,  and hence $\det \left [U^{n-1}e_1 \ U^{n-2}e_1 \ldots U e_1 e_1 \right ] =0$. However, by \cite[Lemma 3.3]{GSVstaircase}, 
\[
\phhi_1(X,Y) = \pm \left (\det Y\right )^{n-1} \det \left [U^{n-1}e_1 U^{n-2}e_1 \cdots U e_1\ e_1 \right ]
\]
and so, whenever $\det Z$ vanishes so does $\phhi_1$. This proves the first claim of the Lemma.

To establish the second claim, let us consider the dependence of the coefficient matrix $Z$ and functions $\bar c_j=
c_j(\bar X,Y)$ on the initial cluster variables. In view of the proof of Proposition \ref{matrixentries}, all $z_{ij}$ 
and $\bar c_j$ are Laurent polynomials in variables from $\widehat\FFF_n$ and, moreover, are polynomials in $\phhi_1$. 
Write $Z = Z^0 + \phhi_1 Z^1$, where $Z^0$ does not depend on $\phhi_1$. Since $\det Z$ is a scalar multiple of 
$\phhi_1$, the entries of $Z^{-1}$ are Laurent polynomials in cluster variables from $\widehat\FFF_n$. 
Furthermore, it is easy to see that $\det(Z^0 + \phhi_1 Z^1)$ is proportional to $\phhi_1^{n-\rank(Z^0)}$, 
and hence the rank of  $Z^0$ is equal to $n-1$. Further, $Z^{-1} = \phhi_1^{-1} W^0  + W^1$, 
where $W^1$ is polynomial in $\phhi_1$, and  $W^0$ does not depend on $\phhi_1$ and satisfies relations
$Z^0 W^0 = W^0 Z^0=0$. It follows that $W^0$ is of rank $1$, that is,  
$W^0 = w_1 w_2^T$, where $w_1, w_2$ are non-zero column vectors such that $w_1$ spans the kernel of $Z^0$. 
Therefore,  the first row of $X$ can be expressed as
\[
X_{[1]} = \phhi_1^{-1} \left(\sum_{j=1}^n (c_j - \bar c_j) w_{1j}\right) w_2^T + \L, 
\]
where $\L$ is a vector of Laurent polynomials in terms of cluster variables in $\bar\FFF'_n$.

Let us show that the vector
\[
u=\left ( \left ( (-1)^{n-1} h_{22}\phhi_{n+1}\right )^j \phhi_2 ^{n-j}\right )_{j=1}^{n}
\]
spans the kernel of $Z^0$. Note that \eqref{longidXY} implies 
\begin{align*}
\sum_{j=1}^n c_j u_{j} + \det Y \phhi_2^n&= \phhi_1 \phhi'_1,\\
\sum_{j=1}^n \bar c_j u_{j} + \det Y \phhi_2^n&= \phhi_1 \bar\phhi'_1,
\end{align*}
where $\bar\phhi'_1$ is a polynomial in the entries of $X_{[2,n]}$ and $Y_{[2,n]}$, and thus can be written as a Laurent polynomial in variables from $\widehat\FFF_n$ which is  polynomial in $\phhi_1$, and hence as a Laurent polynomial in variables from $\bar\FFF'_n$. Consequently,
$\sum_{j=1}^n( c_j - \bar c_j) u_{j} = \phhi_1 ( \phhi'_1 - \bar\phhi'_1 )$, 
so that
\[
\phhi_1 ( \phhi'_1 - \bar\phhi'_1 )=X_{[1]}Zu=X_{[1]}Z^0u+\phhi_1X_{[1]}Z^1u,
\] 
and hence $\phhi_1=0$ implies $X_{[1]}Z^0u=0$. Note that $Z^0$ and $u$ does not depend on $\phhi_1$ and on $X_{[1]}$, 
which means that $u$ spans the kernel of $Z^0$, as claimed. Therefore, we can choose $w_1=u$, and since the entries of 
$u$ are monomials in terms of variables from $\bar\FFF_n\setminus\{\phhi_1\}$, entries of $w_2$ are Laurent monomials
in terms of the same variables. 
We conclude that 
\[
X_{[1]} = \phhi_1^{-1} \left (\sum_{j=1}^n(c_j - \bar c_j) u_j\right ) w_2^T + \L =  
\phhi'_1 w_2^T + ( \L - \bar\phhi'_1 w_2^T ), 
\]
which proves the claim.
\end{proof}

It follows immediately from Lemma~\ref{firstrow} that each $x_{1i}$ can be written as a Laurent polynomial in terms of the
cluster variables in the initial cluster and in any of its neighbours, since by Proposition \ref{matrixentries}, all entries 
of $Y$ and $X_{[2,n]}$ are cluster variables, and hence Laurent polynomials in any cluster.

\section{Two generalized cluster structures on $D(GL_4)$}\label{twogcs}

 In \cite{GSVDouble} we described a generalized cluster structure $\GCC_n^D$ on $D(GL_n)$. 
It is easy to see that the generalized cluster structures $\NGCC_n^D$ described in this paper and $\GCC_n^D(Y^T,X^T)$
have the same set of frozen variables. Moreover, for $n=2$ the initial seeds $\bar\Sigma_2(X,Y)$ and $\Sigma_2(Y^T,X^T)$
coincide. For $n=3$, a straightforward computation shows that the sequence of mutations at vertices $(4,3)$, $(3,2)$, $(2,1)$ takes $\bar\Sigma_3(X,Y)$ to $\Sigma_3(Y^T,X^T)$. In this section we prove that for $n=4$ no such sequence exists, and hence generalized cluster structures $\NGCC_4^D$ and $\GCC_4^D(Y^T,X^T)$ are distinct. We conjecture that this holds
true for any $n>4$ as well, see Remark~\ref{alln} below.

\begin{proposition}\label{notequiv}
The generalized cluster structures $\NGCC_4^D=\NGCC_4^D(X,Y)$ and $\GCC_4^D=\GCC_4^D(Y^T,X^T)$ are distinct.
\end{proposition}  

\begin{proof}
We start with the seed $\bar\Sigma_4$ for $\NGCC_4^D$ and perform a sequence of mutations at vertices 
$(5,4)$, $(4,3)$, $(3,2)$, $(6,4)$, $(5,3)$, $(4,2)$, $(5,4)$, 
$(4,3)$. In this way, we get functions $\phhi_4'$, $\phhi_3'$, $\phhi_2'$, $\phhi_8'$, 
$\phhi_7'$, $\phhi_6'$, $\phhi_4''$, $\phhi_3''$, 
respectively. A straightforward computation shows that
\begin{gather*}
\phhi_8'=\det\begin{bmatrix} y_{43} & y_{44}\\
                      x_{43} & x_{44}
              \end{bmatrix}, \quad
\phhi_7'=\det\begin{bmatrix} y_{42} & y_{43} & y_{44}\\
                             x_{32} & x_{33} & x_{34}\\
                             x_{42} & x_{43} & x_{44}
              \end{bmatrix}, \quad		
\phhi_4''=\det\begin{bmatrix} y_{32} & y_{33} & y_{34}\\
                              y_{42} & y_{43} & y_{44}\\
                              x_{42} & x_{43} & x_{44}
              \end{bmatrix},	\\						
\phhi_6'=\det\begin{bmatrix} y_{31} & y_{32} & y_{33} & y_{34}\\
                             y_{41} & y_{42} & y_{43} & y_{44}\\
                             x_{31} & x_{32} & x_{33} & x_{34}\\
                             x_{41} & x_{42} & x_{43} & x_{44}
              \end{bmatrix}, \quad
\phhi_3''=\det\begin{bmatrix} y_{21} & y_{22} & y_{23} & y_{24}\\
                              y_{31} & y_{32} & y_{33} & y_{34}\\
                              y_{41} & y_{42} & y_{43} & y_{44}\\  
                              x_{41} & x_{42} & x_{43} & x_{44}
              \end{bmatrix}.							
\end{gather*}
The corresponding quiver is shown in Fig.~\ref{modiQ4}. 
The quiver for the initial seed for $\GCC_4^D$ constructed in \cite{GSVDouble}
is shown in Fig.~\ref{oldQ4}.  Recall that we are interested in the seed $\Sigma_4(Y^T,X^T)$, and hence in this case $g_{ij}(Y^T,X^T)=\det (Y^T)_{[i,n]}^{[j,j+n-i]}=h_{ji}(X,Y)$ and $h_{ij}(Y^T,X^T)=g_{ji}(X,Y)$. Further, functions 
$f_{ij}$ are defined via 
\[
f_{ij}(Y^T,X^T)=\det \left[(Y^T)^{[n-i+1,n]}\ (X^T)^{[n-j+1,n]}\right]_{[n-i-j+1,n]}, 
\]
and hence $f_{11}(Y^T,X^T)=\fy'_8(X,Y)$, $f_{21}(Y^T,X^T)=\fy''_4(X,Y)$, and $f_{12}(Y^T,X^T)=\fy'_7(X,Y)$. Finally, as explained in \cite[Remark 3.1]{GSVDouble},
functions $\fy_{ij}$ with $i+j=4$ are defined via the same expression as $f_{ij}$, and hence $\fy_{31}(Y^T,X^T)=\fy''_3(X,Y)$, 
$\fy_{22}(Y^T,X^T)=\fy'_6(X,Y)$, and $\fy_{13}(Y^T,X^T)=\fy_9(X,Y)$. We thus see
that the restrictions of both quivers to the three lower rows coincide,
as well as the functions attached to the corresponding vertices. Moreover, the functions attached to the fourth
row from below in both quivers coincide as well, as well as the arrows between the third and the fourth row.
We will prove that the corresponding two seeds
are not mutationally equivalent.

\begin{figure}[ht]
\begin{center}
\includegraphics[width=11cm]{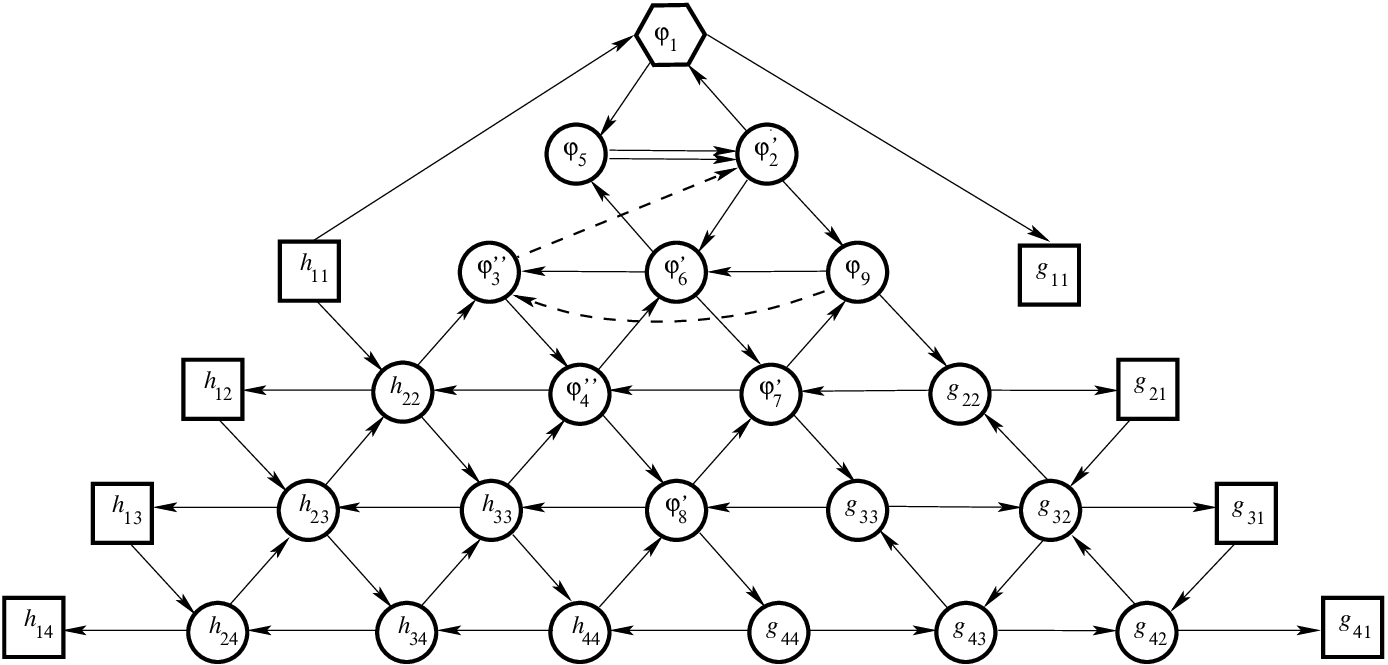}
\end{center}
\caption{Modified $\bar Q_4$}
\label{modiQ4}
\end{figure}

\begin{figure}[ht]
\begin{center}
\includegraphics[width=11cm]{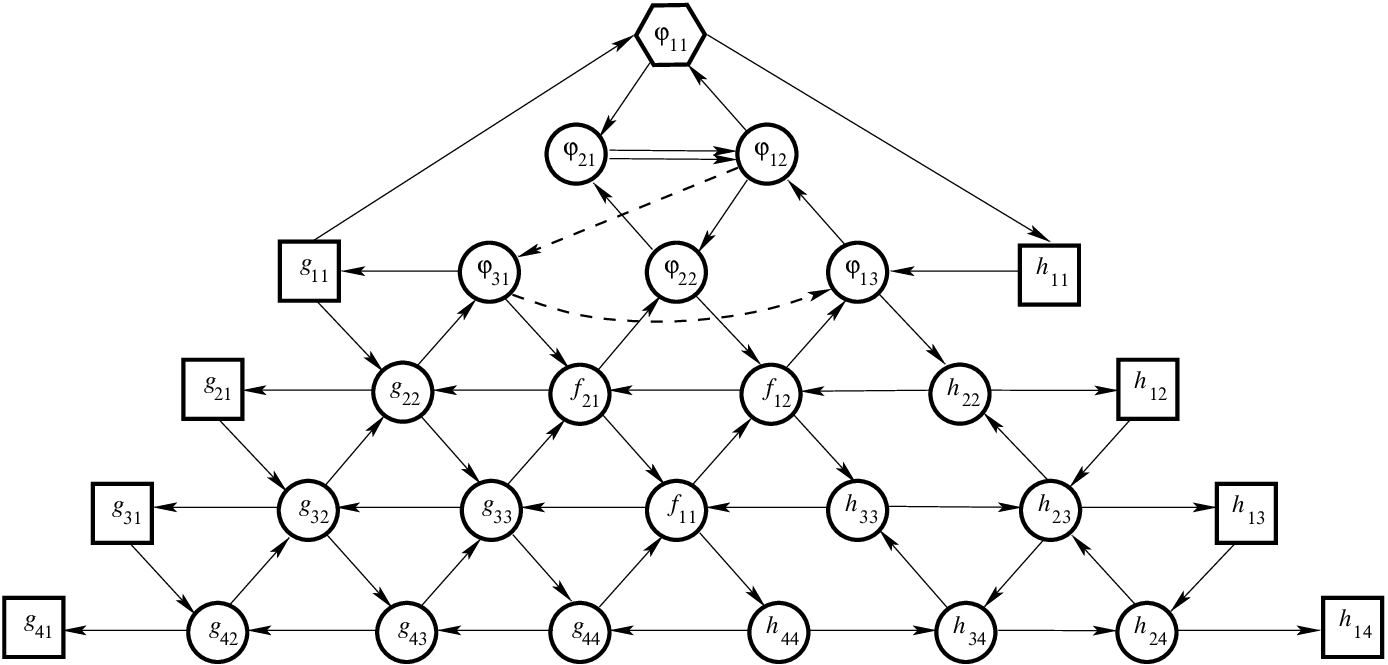}
\end{center}
\caption{Quiver $ Q_4$}
\label{oldQ4}
\end{figure}

 By \cite[Th.~3.6]{CL2}, if two seeds are mutationally equivalent and share a set of common cluster variables, there exists a sequence of mutations that connects these seeds and does not involve the common cluster variables. 

\begin{remark}
In fact, the definition of
a generalized cluster structure in \cite{CL2} and in the preceding paper \cite{CL1} is more restrictive, since it imposes
a reciprocity condition on exchange coefficients, following \cite{Na}. However, this condition is only used in the proof
of Lemma 4.20 in \cite{CL1}, which in turn is based on Proposition 3.3 in \cite{Na}. This proposition claims that every cluster
variable can be written as a Laurent polynomial in cluster variables of the initial cluster and an ordinary polynomial in frozen
variables. It is an analog of the corresponding statement for ordinary cluster structures and its proof extends to the case of generalized cluster structures as defined in Section \ref{prelim} without any changes. 
\end{remark}

Consequently, if the above two seeds are
equivalent, there should exist a sequence of mutations that involves only three vertices comprising the uppermost triangle.
We will concentrate on two four-vertex subquivers that are formed by the uppermost triangle and the vertex corresponding
to  $\fy_{13}(Y^T,X^T)=\fy_9(X,Y)$. These two subquivers are shown in Fig.~\ref{two4quivers}. We claim that there is no sequence of mutations at
the vertices $1$, $2$ and $3$ that takes one subquiver to the other one. Note that  although the mutations at vertex $4$ are not allowed, it is not frozen.

\begin{figure}[ht]
\begin{center}
\includegraphics[width=7cm]{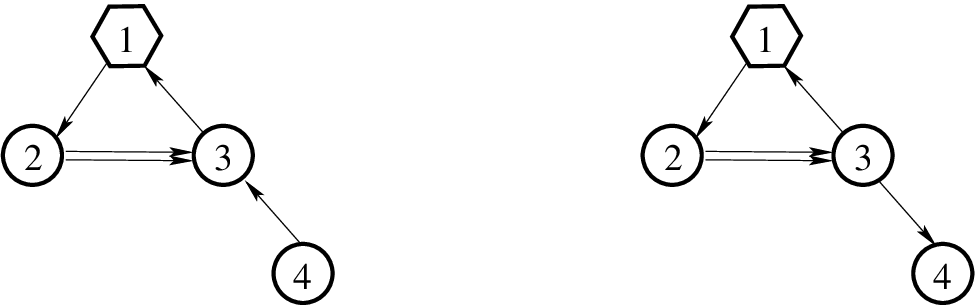}
\end{center}
\caption{Subquivers of $Q_4$ and $\bar Q_4$}
\label{two4quivers}
\end{figure}

To prove our claim we consider the evolution of a more general quiver $Q(\alpha,\beta,\gamma)$ shown in Fig.~\ref{samplequiv} under mutations at the vertices $1$, $2$ and $3$. Here multiplicities $\alpha$, $\beta$ and $\gamma$ can take any integer values except for $\alpha=\beta=\gamma=0$. A negative value 
means that the direction of the corresponding arrow is reversed. Clearly, two quivers shown in  Fig.~\ref{two4quivers} are
$Q(0,1,0)$ and $Q(0,-1,0)$. To keep track of the mutations it will be convenient to renumber the vertices so that the generalized vertex is always vertex $1$, and the direction of arrows in the triangle is
$1\to2\to3\to1$. Note that any mutation of $Q(\alpha,\beta,\gamma)$ transforms it into  $Q(\alpha',\beta',\gamma')$
for certain values of $\alpha'$, $\beta'$ and $\gamma'$. 

\begin{figure}[ht]
\begin{center}
\includegraphics[width=2cm]{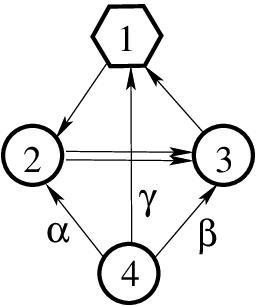}
\end{center}
\caption{Quiver $Q(\alpha,\beta,\gamma)$ }
\label{samplequiv}
\end{figure}

Define the {\it charge\/} of $Q=Q(\alpha,\beta,\gamma)$ as $C(Q)=|\alpha|+|\beta|+|\gamma|$. The nodes of the 3-regular tree
$\T$ that describes all possible mutations of $Q(\alpha,\beta,\gamma)$ can be classified into 10 possible types with respect to
the charge. We encode these types by a triple $[i,j,k]$, where $i$ stands for the number of mutations that increase the
charge, $j$ stands for the number of mutations that preserve the charge, and $k$ stands for the number of mutations that decrease the charge, so that $i+j+k=3$. Note that both quivers $Q(0,1,0)$ and $Q(0,-1,0)$ have charge $1$. Consequently, if they are
mutation equivalent, then either all quivers along the simple  path (the one that never returns to the same vertex) in $\T$ that connects $Q(0,1,0)$ and $Q(0,-1,0)$ have charge $1$, or
this path contains two quivers $Q_1$, $Q_2$ that differ by one mutation, such that $C(Q_1)<C(Q_2)$ and $C(Q_2)\ge C(Q)$ for any other quiver $Q$ along the path.

Consider the first possibility. A straightforward computation shows that mutations at vertices $1$ and $2$ 
preserve the charge and take $Q(0,1,0)$ to $Q(1,0,0)$, while mutation at vertex $3$ increases the charge. Further, 
mutations of $Q(1,0,0)$ at vertices $1$ and $3$ preserve the charge and take $Q(1,0,0)$ to $Q(0,1,0)$, while mutation at  vertex $2$ increases the charge. Therefore, $Q(0,-1,0)$ can not be reached from $Q(0,1,0)$ along a path with the constant charge $1$.

Consider the second possibility. 
 Let $[i,j,k]$ be the type of $Q_2$, then $k\ge 1$ and $j+k\ge2$, so we remain with the following possibilities: 
$[1,1,1]$, $[1,0,2]$, $[0,2,1]$, $[0,1,2]$, $[0,0,3]$. We are interested in finding conditions on $\alpha$, $\beta$ and 
$\gamma$ that would guarantee that the type of $Q_2$ is indeed one of the types listed above. One can distinguish 
eight possible cases according to the signs of $\alpha$, $\beta$, and $\gamma$. Let us consider in detail one of the nontrivial cases.

\begin{table}[htp]
\caption{Types of $Q$}
\begin{center}
\begin{tabular}{|c||c|c|}
\hline
 & {\it Case 1\/}: $\alpha\ge 0,\beta\ge 0, \gamma\ge 0$ &  {\it Case 2\/}: $\alpha<0,\beta\ge 0,\gamma\ge 0$ \\
 \hline
$C(\mu_1(Q))$  & $C(Q)+4\gamma\ge C(Q)$ &  $C(Q)+\alpha+|\alpha+4\gamma|$ \\
 \hline
$C(\mu_2(Q))$  & $C(Q)+2\alpha\ge C(Q)$ &   $C(Q)-\gamma+|\alpha+\gamma|$ \\
\hline
$C(\mu_3(Q))$  & $C(Q)+\beta\ge C(Q)$ &  $C(Q)+\beta$ \\
\hline
type of $Q$ & $[i,j,0]$ & 
$\begin{cases}
[1,1,1], &\text{ if } \alpha+2\gamma>0, \beta=0 \\ &\text{ or if } \alpha+2\gamma<0, \\ & \gamma\ne 0, \beta=0 \\ [i,j,0] \text{ or }  & \\
[2,0,1] & \text{ otherwise.}  
\end{cases}$ \\
\hline
\hline
\hline
 &  {\it Case 3\/}: $\alpha\ge 0,\beta < 0, \gamma\ge 0$ &  {\it Case 4\/}: $\alpha\ge 0,\beta\ge 0, \gamma< 0$\\
 \hline
$C(\mu_1(Q))$   & $C(Q)+4\gamma\ge C(Q)$ & $C(Q)-\beta+|\beta+4\gamma|$ \\
 \hline
$C(\mu_2(Q))$   & $C(Q)+\beta+|\beta+2\gamma|$ & $C(Q)+2\alpha\ge C(Q)$ \\
\hline
$C(\mu_3(Q))$  & $C(Q)-\alpha+|\alpha+2\beta|$   & $C(Q)+\gamma+|\gamma+\beta|$\\
\hline
type of $Q$  & 
$\begin{cases}
[1,1,1], &\text{ if } \alpha+\beta>0, \gamma=0 \\ &\text{ or if } \alpha+\beta<0,\\ & \alpha\ne 0, \gamma=0 \\ [i,j,0] \text{ or }  & \\
[2,0,1] & \text{ otherwise.} 
 \end{cases}$ & $\begin{cases}
[1,1,1], &\text{ if } \beta+2\gamma>0, \alpha=0 \\ &\text{ or if } \beta+2\gamma<0, \\ & \gamma\ne 0, \alpha=0 \\ [i,j,0] \text{ or }  & \\
[2,0,1] & \text{ otherwise.}  
\end{cases}$  \\
\hline
\hline
\hline
 &  {\it Case 5\/}: $\alpha<0,\beta< 0, \gamma\ge 0$ & {\it Case 6\/}: $\alpha< 0,\beta \ge 0, \gamma< 0$ \\
 \hline
$C(\mu_1(Q))$  &  $C(Q)+\alpha+|\alpha+4\gamma|$ &  $C(Q)-\beta+|\beta+4\gamma|$ \\
 \hline
$C(\mu_2(Q))$  &   $C(Q)-\gamma+|\alpha+\gamma|$ &  $C(Q)-\alpha>C(Q)$ \\
\hline
$C(\mu_3(Q))$   &  $C(Q)-2\beta>C(Q)$ &  $C(Q)+\gamma+|\gamma+\beta|$ \\
\hline
type of $Q$ & 
$[i,j,0]\text{ or } [2,0,1]$
&
$[i,j,0]\text{ or } [2,0,1]$\\
\hline
\hline
\hline
 &  {\it Case 7\/}: $\alpha\ge0,\beta< 0, \gamma< 0$ & {\it Case 8\/}: $\alpha< 0,\beta < 0, \gamma< 0$ \\
 \hline
$C(\mu_1(Q))$  &  $C(Q)-4\gamma>C(Q)$ & $C(Q)-4\gamma>C(Q)$ \\
 \hline
$C(\mu_2(Q))$  &   $C(Q)+\beta+|\beta+2\alpha|$ & $C(Q)-\alpha>C(Q)$ \\
\hline
$C(\mu_3(Q))$   &   $C(Q)-\alpha+|\alpha+2\beta|$ & $C(Q)-2\beta>C(Q)$ \\
\hline
type of $Q$ & 
$[i,j,0]\text{ or } [2,0,1]$
&
[3,0,0] \\
\hline

\end{tabular}
\end{center}
\label{table:typesQ}
\end{table}%

Let $\alpha<0$, $\beta\ge0$, $\gamma\ge0$. Mutation at vertex $1$ takes $Q=Q(\alpha,\beta,\gamma)$ to
$Q'=Q(\beta,\alpha+4\gamma, -\gamma)$, so that $C(Q')=C(Q)+\alpha+|\alpha+4\gamma|$, and hence $C(Q')<C(Q)$
if $\alpha+2\gamma<0$ and $\gamma\ne0$. Mutation at vertex $2$ takes 
$Q$ to $Q'=Q(\beta,-\alpha, \gamma+\alpha)$, so that $C(Q')=C(Q)-\gamma+|\alpha+\gamma|$, and hence $C(Q')<C(Q)$
if $\alpha+2\gamma>0$. Mutation at vertex $3$ takes 
$Q$ to $Q'=Q(-\beta,\alpha, \gamma+\beta)$, so that $C(Q')=C(Q)+\beta\ge C(Q)$. Consequently, the type of
$Q(\alpha,\beta,\gamma)$ is $[1,1,1]$ if $\alpha+2\gamma>0$ and $\beta=0$, or if $\alpha+2\gamma<0$, $\gamma\ne0$ and 
$\beta=0$. For the other values of parameters, the type of
$Q(\alpha,\beta,\gamma)$ is either $[i,j,0]$  with $i+j=3$ or $[2,0,1]$.

Results of similar considerations in all the remaining  cases are summarized in Table~\ref{table:typesQ} 
that contains,  for each case, the values of the charge after the three possible mutations and possible types of 
$Q$ depending on the values of $\alpha$, $\beta$, and $\gamma$.

 It follows from the results presented in the table that the only possible candidates for the quiver $Q_2$ 
are quivers of type $[1,1,1]$ in Cases 2, 3 and 4. 
In all three cases the next node in the path should have the same charge.
If $Q_2$ is as in Case 2 then the next node is obtained by mutation 
at vertex $3$, and the resulting quiver is $Q'=Q(0,\alpha,\gamma)$ with $\alpha<0$ and $\gamma>0$. This situation is covered
by Case 3, and the other two mutations of $Q'$ yield $C(Q')+4\gamma>C(Q')$ and $C(Q')+|2\alpha|>C(Q')$. Consequently,
the maximality condition for the charge of $Q_2$ fails.

 The remaining two cases are analyzed in a similar way, with Case 3 leading to Case 2 and Case 4 leading to Case 4.
Therefore,  $Q(0,-1,0)$ can not be reached from $Q(0,1,0)$ along a path with a varying charge, which completes the proof.
\end{proof}

\begin{remark}  
In a recent preprint \cite{Zhou}, a technique of scattering diagrams is used to construct a log Calabi--Yau variety with two nonequivalent cluster structures both associated with the Markov quiver. This variety is obtained by a certain 
augmentation of a cluster A-variety with principal coefficients in the sense of \cite{GHKK}. In contrast to this, 
the example we presented above gives two explicitly defined nonequivalent generalized cluster structures with the same set of frozen variables in an affine variety obtained by deleting a divisor with a normal crossing from an affine space.
\end{remark}

\begin{remark}\label{alln}
For $n>4$ one can build a sequence of mutations that takes the seed $\bar\Sigma_n$ to a seed $\Sigma=(\x,Q,\P)$ with the following properties. Let $\Sigma_n(Y^T,X^T)=(\x_n,Q_n,\P_n)$, and assume that $Q_n$ and $Q$ are arranged in $2n-2$ rows, 
as in Figs.~\ref{modiQ4} and~\ref{oldQ4},
then the restrictions of $Q$ and $Q_n$ to $n-1$ lower rows coincide, as well as the functions attached to the corresponding vertices. 
So do the functions attached to the $n$th row from below in both quivers and
the arrows between the $n$th and the $(n-1)$th row. The restrictions of
$Q$ and $Q_n$ to the remaining $n-2$ upper rows also coincide, however,  the functions
attached to the corresponding vertices differ. Finally, $\phhi_{n-2}\to\phhi_{N-n+1}$ is an edge in $Q$, 
$\phhi_{1,n-1}\to\phhi_{1,n-2}$ is
an edge in $Q_n$, $\phhi_{N-n+1}(X,Y)=\phhi_{1,n-1}(Y^T,X^T)$, vertices $\phhi_{n-2}$ and $\phhi_{1,n-2}$ correspond each other in the upper parts of $Q$ and $Q_n$, respectively, and there are no edges between $\phhi_{N-n+1}$ and the upper part of $Q$ ($\phhi_{1,n-1}$ and the upper part of $Q_n$, respectively). We believe that similarly to the case $n=4$, one can prove that it is impossible to invert the arrow in question via a sequence of mutations at the vertices of the upper part, which would imply that the generalized cluster structures $\NGCC_n^D(X,Y)$ and $\GCC_n^D(Y^T,X^T)$ are distinct.
\end{remark}

The example presented above describes two different generalized cluster structures  
$\GCC_1=\NGCC_4^D$, $\GCC_2=\GCC_4^D$
such that 
\begin{itemize}
\item the corresponding upper cluster algebras coincide with the ring $\mathbb{C}[D(GL_4)]$ of regular functions on 
the Drinfeld double of $GL_4$;
\item both generalized cluster algebras are compatible with the standard Poisson--Lie bracket on $D(GL_4)$
and have the same collection of frozen cluster variables.  
\end{itemize}

We believe that both generalized cluster structures $\GCC_1$ and $\GCC_2$ can be related to  the same ordinary cluster 
structure $\CC$ using a conjectural construction outlined below  for the case of general $n$.

The cluster structure $\CC$ is associated with the moduli space ${\mathcal A}_{G,S}$ introduced by Fock and Goncharov in the study of $G$-local systems on a marked Riemann surface $S$, see \cite{FG}.  In our example $G=GL_n$, $S$ is  
the punctured disk with $4$ marked points on the boundary.
The variety ${\mathcal A}_{GL_n,S}$ is homeomorphic to the configuration space of triples $({\bf F}, M,\Phi)$ modulo 
the $GL_n$-action where ${\bf F}$ denotes a quadruple of decorated flags at the marked points, $M$ is the 
$GL_n$-monodromy about the puncture, $\Phi$ is a flag at the puncture which is  invariant under the monodromy.
Note that the invariant flag at the puncture is not uniquely defined by the monodromy: there are $n!$ choices 
corresponding to different orderings of monodromy eigenvalues.
The Weyl group $W=S_n$ acts on ${\mathcal A}_{GL_n,S}$ by reordering eigenvalues, which results in a different choice 
of the invariant flag.

The parametrization of ${\mathcal A}_{GL_n,S}$ introduced in \cite{FG} endows it with a cluster structure which, 
in turn, leads to a compatible Poisson bracket. 
This Poisson structure has corank $2n$ with $n$ Casimirs given by the coefficients of the characteristic polynomial of
the monodromy  and $n$ additional Casimirs.
Fixing values of $n$ additional Casimirs to $1$, we obtain a codimension $n$ subvariety $V$. 
The action of $W$ restricts to $V$.
Further, $V$ is a cluster variety whose coordinate ring is equipped with the cluster structure $\CC$. 
In particular, $\CC$ inherits the $W$-action. 

There is a natural projection  $\pi:V\to D(GL_n)$  
with a fiber $W\times H$ where $H$ is the Cartan subgroup.
We conjecture that the projection $\pi$ provides a natural connection between generalized cluster structures $\GCC_i$ in 
${\mathbb C}[ D(GL_n)]$ for $i=1,2$ and $\CC$. More precisely, the pullbacks of all cluster variables in $\GCC_i$ are 
$W$-invariant cluster variables in $\CC$. 
Furthermore, each seed $\Sigma$ of the generalized cluster structure 
$\GCC_i$ contains  one cluster variable $g(\Sigma)$ attached to the special vertex of the quiver and satisfying a 
generalized mutation rule, while the remaining cluster variables  obey the usual mutation rules. 
The rank of  $\CC$ is $n-1$ more than the rank of $\GCC_i$. Any seed $\Sigma$ of $\GCC_i$ corresponds 
to a seed $\tilde \Sigma$ of $\CC$ in which $g(\Sigma)$ corresponds to an $n$-tuple of cluster variables 
$\tilde g_j(\tilde \Sigma)$, $j\in[1,n]$, such that $\pi^*g(\Sigma)=\prod_{j=1}^n \tilde g_j(\tilde \Sigma)$.
The remaining cluster variables of the seed $\tilde \Sigma$ and frozen variables of $\CC$ are obtained as pullbacks 
$\pi^*$ of the corresponding cluster variables of $\Sigma$ and frozen variables in $\GCC_i$.
The generalized mutation of $g(\Sigma)$ corresponds to the composition $s(\tilde \Sigma)$ of $n$ mutations at all 
$\tilde g_j(\tilde\Sigma)$ taken in any order (mutations of $\tilde g_j(\tilde\Sigma)$ commute). Namely, let
$g'$ denote the function obtained by the generalized mutation of  $g(\Sigma)$   and  
$\tilde g'_j=s(\tilde \Sigma)(\tilde g_j(\tilde\Sigma))$, then 
$\pi^*g'=\prod_{j=1}^n \tilde g'_j$. Further, the set of seeds of $\CC$ corresponding to $\GCC_1$ is disjoint from 
the set of seeds corresponding to $\GCC_2$.

A detailed proof will be presented elsewhere.

\begin{remark}
Let $G$ be a semisimple  complex Lie group with the Lie algebra $\mathfrak g$.
The group $G$ is equipped with the standard Poisson-Lie structure. In~\cite{Shen} the semiclassical limit of ${\mathcal U}_q({\mathfrak g})$ is realized as a quotient by an ideal generated by Poisson central elements
of the $W$-invariant subring of the coordinate ring of the second moduli space 
 ${\mathcal X}_{G,S_{0,1,2}}$ of Fock--Goncharov cluster ensemble, where  $S_{0,1,2}$ is a once punctured disk with two marked points on the boundary. This construction seems to be closely related to the projection $\pi$ above. However, no cluster structure on the $W$-invariant subring was considered in ~\cite{Shen}. 
\end{remark}

\section{Reduction to a generalized cluster structure on band periodic matrices}\label{bandmat}

In \cite{GSVstaircase} we presented a framework for constructing  generalized cluster structures. It is  based on certain identities associated with periodic staircase shaped matrices. One  of the examples considered in
\cite{GSVstaircase} was the generalized cluster structure $\NGCC_n^D$ that we treated in previous sections. 
Another example was a generalized cluster structure on the space 
of $(k+1)$ diagonal $n$-periodic band matrices with $k\leq n$. In this section, we will show how the latter structure, denoted here by $\GCC(\L_{kn})$ can be obtained as a restriction of the former. In particular, this will allow us to obtain an analogue of Theorem \ref{newstructure} for $\GCC(\L_{kn})$.

\subsection{Initial cluster}\label{bandinit}
In the case of the Drinfeld double $D(GL_n)$, the periodic staircase matrix mentioned above is an infinite block bidiagonal matrix 
\begin{equation}
\label{shapeL}
L= \left [
\begin{array}{ccccccc}
\ddots &\ddots &\ddots &\ddots & & &\\
 & 0 & X & Y & 0 &  &\\
& & 0 & X & Y & 0 &\\
 & & &\ddots &\ddots &\ddots &\ddots
 \end{array}
\right ],
\end{equation}
that corresponds to $(X,Y)\in D(GL_n)$.  Now we drop the invertibility requirement for $X$ and 
 choose $Y$ to be a lower triangular band matrix with  min$(k+1,n)$ non-zero diagonals 
(including the main diagonal) and $X$ to be  a
 matrix with zeroes everywhere outside of the $k\times k$ upper triangular block in the upper right corner:
\begin{equation}
\label{XYband}
\begin{aligned}
X &= \left [
\begin{array}{cccccc}
0 & \cdots & 0 & a_{11} & \cdots & a_{k1}\\
0 & \cdots& 0 & 0 & a_{12} & \cdots \\
\vdots & \vdots & \vdots & \vdots & \ddots& \ddots \\
0 & \cdots & 0 & \cdots  & 0 & a_{1k}\\
0 & \cdots & 0 & \cdots & \cdots & 0\\
\vdots & \vdots & \vdots & \vdots & \vdots& \vdots \\
\end{array}
\right ],\\
Y&=
\left [
\begin{array}{cccccc}
a_{k+1,1} & 0 & \cdots & \cdots& \cdots & 0\\
a_{k2} & a_{k+1,2}  & 0 & \cdots & \cdots & \cdots \\
\vdots & \vdots & \ddots & \vdots & \vdots& \vdots \\
a_{1,k+1}& a_{2,k+1}& \cdots  & a_{k+1,k+1} & 0 &\cdots\\
0  & \ddots & \ddots & \vdots & \ddots& \vdots \\
0 & \cdots & a_{1n} & a_{2n} &\cdots & a_{k+1,n}
\end{array}
\right ] \quad\text{for $k<n$},\\
Y&= \left(a_{n+1-i+j,i} \right )_{i,j=1}^n\quad\text{for $k=n$},
\end{aligned}
\end{equation} 
where we assume that $a_{ti}=0$ when $t>k+1$. Then $L$ in \eqref{shapeL} is  a $(k+1)$-diagonal $n$-periodic band matrix.
We denote by $\L_{kn}$ the space of such matrices with an additional condition that all entries of the lowest and the 
highest diagonals are nonzero. 
Let $\bar \L_{kn}$ be the closure of $\L_{kn}$ in the space $D(\Mat_n)=\Mat_n\times \Mat_n$. Every element of 
$\bar\L_{kn}$ is identified with a pair of matrices of the form \eqref{XYband}. In particular, vanishing of the lowest 
diagonal yields an inclusion $\L_{k-1,n}\subset\bar\L_{kn}$.

Note that when such matrices are substituted into \eqref{Phi}, $\Phi$ becomes reducible with a leading irreducible block 
$\Phi^{(k)}$ of  size $(k-1)(n-1)\times(k-1)(n-1)$. For $i=1,\ldots, (k-1)(n-1)$ we define
\begin{equation}
\label{tildephi}
\tilde\phhi_i=\tilde\phhi_i^{(k)} = \det\Phi_{[i,(k-1)(n-1)]}^{[i,(k-1)(n-1)]}.
\end{equation}

By \cite{GSVstaircase}, a generalized cluster structure in the space of regular functions on $\L_{kn}$ is defined by the following data.

Define the family $\FFF_{kn}$ of $(k+1)n$ regular functions  on $\L_{kn}$ via 
\begin{equation}
\label{bandseed}
\FFF_{kn}=\left\{\{\tilde\phhi_{i}^{(k)} \}_{i=1}^{(k-1)(n-1)};\ \tilde a_{11};\ \{a_{1i}\}_{i=2}^n;\ \{a_{k+1,i}\}_{i=1}^n;\
\{\tilde c_i(X,Y)\}_{i=1}^{k-1}\right\},
\end{equation}
where $\tilde a_{11}=(-1)^{k(n-1)}a_{11}$ and $\tilde c_i(X,Y)=(-1)^{i(n-1)}c_i(X,Y)$ for $1\le i\le k-1$ with $c_i(X,Y)$
satisfying the identity $\det(\lambda Y +\mu X)=\lambda^{n-k}\sum_{i=0}^k c_i(X,Y)\mu^i\lambda^{k-i}$.

Let $Q_{kn}$ be the quiver with $(k+1)n$ vertices, of which $k-1$ vertices are isolated and are not shown in the figure below, 
$(k+1)(n-1)$ are arranged in an $(n-1)\times (k+1)$ grid and denoted
$(i,j)$, $1\le i\le n-1$, $1\le j\le k+1$, and the remaining two are placed on top of the leftmost and the rightmost 
columns in the grid and denoted $(0,1)$ and $(0,k+1)$, respectively.  All vertices in the leftmost and in the 
rightmost columns are frozen. The vertex $(1,k)$ is special, and its multiplicity equals $k$. All other vertices are regular mutable vertices.

The edge set of $Q_{kn}$
consists of the edges $(i,j)\to (i+1,j)$ for $i=1,\ldots, n-2$, $j= 2,\ldots, k$; 
$(i,j)\to (i,j-1)$ for $i=1,\ldots, n-1$, $j= 2,\ldots, k$, $(i,j)\ne (1,k)$; 
$(i+1,j)\to (i,j+1)$ for $i=1,\ldots, n-2$, $j= 2,\ldots, k$, shown by solid lines. 
In addition, there are edges  $(n-1,3)\to (1,2)$, $(1,2)\to (n-1,4)$,
$(n-1,4)\to (1,3),\ldots, (1,k-1)\to(n-1,k+1)$ that form a directed path (shown by dotted lines). 
Save for this path, and the missing edge 
$(1,k)\to (1,k-1)$, mutable vertices of $Q_{kn}$ form a mesh of consistently oriented triangles

Finally, there are edges between the special vertex $(1, k)$ and frozen vertices $(i,1)$, $(i,k+1)$ for $i=0,\ldots n-1$. 
There are $k-1$ parallel edges between $(1,k)$ and $(i,k+1)$ for $i=1,\dots,n-1$, 
and one edge between $(1,k)$ and all other frozen vertices 
(including $(0,k+1)$).  The edge to $(0,k+1)$ is directed from $(1,k)$; if
$k>2$, all other edges are directed towards $(1, k)$, and if $k=2$, the direction
of the edge between $(1,1)$ and $(1,k)$ is reversed. 

Quiver $Q_{47}$ is shown in Figure~\ref{Qkn}. 

\begin{figure}[ht]
\begin{center}
\includegraphics[width=8cm]{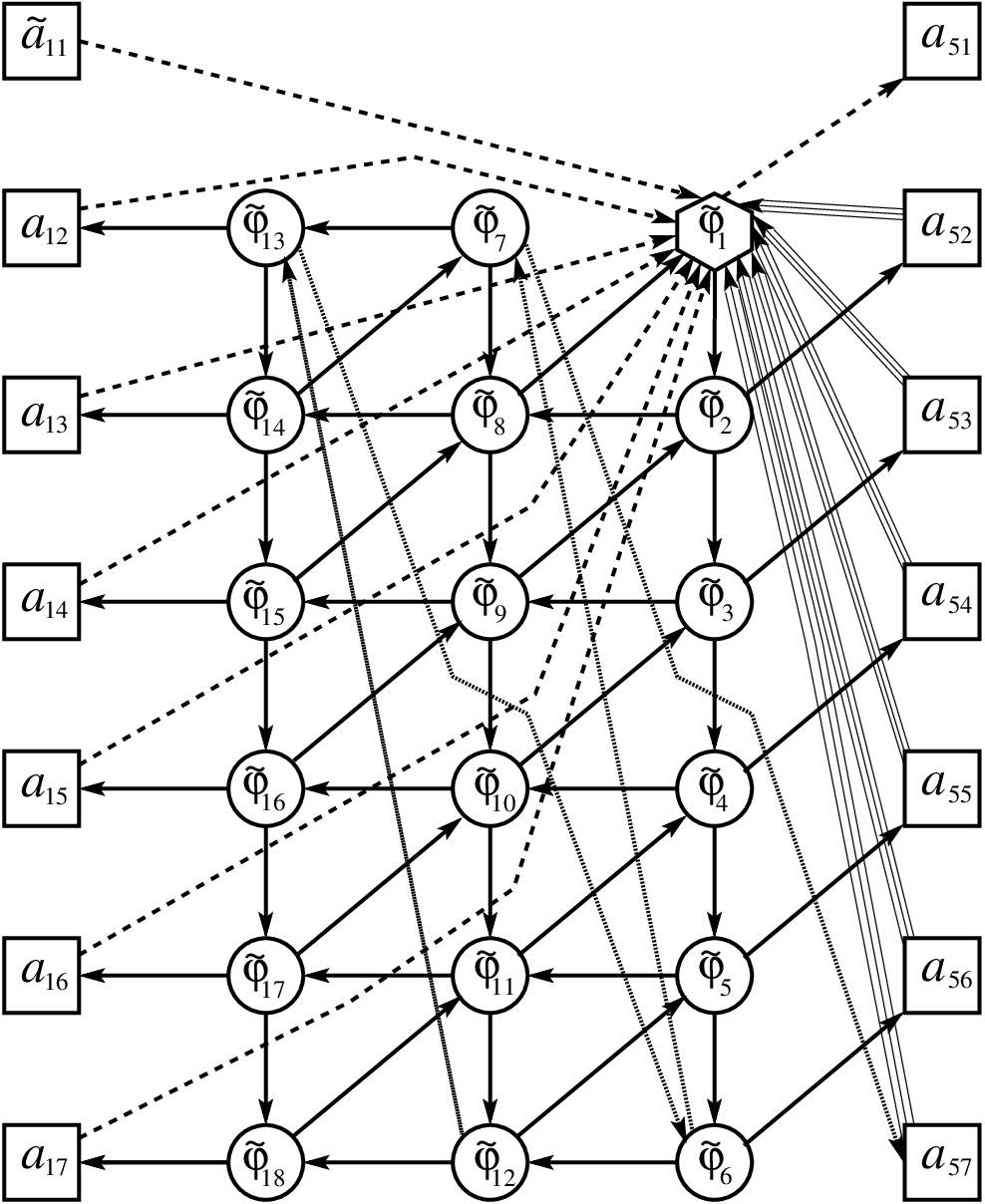}
\end{center}
\caption{Quiver $Q_{47}$}
\label{Qkn}
\end{figure}

We attach functions $\tilde a_{11}, a_{12}, \ldots, a_{1n}$, in a top to bottom order, to the vertices of the leftmost 
column in $Q_{kn}$, and functions $a_{k+1,1},\dots, a_{k+1,n}$, in the same order, to the vertices of the rightmost column
in $Q_{kn}$. Functions $\tilde\phhi_{i}$ are attached, in a 
top to bottom, right to left order, to the remaining vertices of $Q_{kn}$, starting with $\tilde\phhi_{1}$ attached to the special vertex $(1,k)$. The set of strings $\P_{kn}$ contains a unique nontrivial string 
$(1,\tilde c_1(X,Y),\dots,\tilde c_{k-1}(X,Y),1)$ corresponding to the unique special vertex.

\begin{theorem}
\label{Band_structure}
The seed $\Sigma_{kn}=(\FFF_{kn}, Q_{kn},\P_{kn)}$ defines a complete generalized cluster structure  
$\GCC(\Sigma_{kn})$ in the ring of regular functions on $\L_{kn}$ compatible with the restriction of the standard Poisson--Lie structure on $D(\Mat_n)$. 
\end{theorem}

\begin{proof}  The proof follows closely that for Theorem \ref{newstructure}. Regularity of $\GCC(\Sigma_{kn})$ 
is borrowed from \cite[Theorem 5.1]{GSVstaircase}. The proof of log-canonicity and of compatibility is based on the downward
induction on $k$, see
Section \ref{bandcompat} below. The proof of completeness is a modification of a similar statement for the Drinfeld double
and relies on the same ideas, see Section \ref{bandcomplet}.
\end{proof}

\subsection{Compatible Poisson bracket}\label{bandcompat}

Let us re-write the Poisson bracket \eqref{sklyadoubleGL} in terms of matrix entries of a pair of matrices $(X,Y)$:
\begin{equation}
\begin{split}
\label{sklyadoubleXY}
\{ x_{ij}, x_{pq}\} &= \frac 1 2 \left ( \sign (p-i) +  \sign (q - j) \right ) x_{iq} x_{pj},\\
\{ y_{ij}, y_{pq}\} &= \frac 1 2 \left ( \sign (p-i) +  \sign (q - j) \right ) y_{iq} y_{pj},\\
\{ y_{ij}, x_{pq}\} &=  \frac 1 2\left (1+\sign (q - j))  y_{iq} x_{pj} - (1+\sign(i-p)) x_{iq} y_{pj}\right ).
\end{split}
\end{equation}
This Poisson bracket can be extended to $D(\Mat_n)$. It follows from \eqref{sklyadoubleXY} that every inclusion 
in the chain
\[
D(\Mat_n) \supset \bar\L_{nn} \supset \bar \L_{n-1,n} \supset \cdots \supset \bar\L_{kn} \supset \cdots  \supset \bar\L_{2n} 
\]
is a Poisson map. The same is true about inclusions $\L_{kn} \subset \bar\L_{kn}$.

\begin{proposition}
\label{log-can-band}
The family $\FFF_{kn}$ defined in \eqref{bandseed} is log-canonical with respect to the restriction of the Poisson bracket \eqref{sklyadoubleXY} to  $\L_{kn} $.
\end{proposition}
\begin{proof}  
For $k=n$, substitute $(X,Y) \in \L_{nn}$ into \eqref{Phi}. As was mentioned above, $\Phi$ becomes reducible with an irreducible $(n-1)^2 \times (n-1)^2$ upper left block $\Phi^{(n)}$ and, for $ i \leq (n-1)^2$, 
functions $\phhi_i$ restricted to $\L_{nn}$ factor as 
\begin{equation}
\label{factor_n}
\phhi_i = \tilde\phhi^{(n)}_i \phhi_{(n-1)^2+1}.
\end{equation}
 By Theorem \ref{newstructure}, Poisson brackets $\{ \log\phhi_i, \log\phhi_j\} = \omega_{ij}$ are constant on $D(GL_n)$ and, by extension, on $D(\Mat_n)$. Since $\L_{nn}$ is a Poisson submanifold in $D(\Mat_n)$, we obtain
$\{ \log\tilde\phhi^{(n)}_i, \log\phhi_{(n-1)^2+1}\} = \omega_{i,(n-1)^2+1}$ for $ i \leq (n-1)^2$, 
and thus
\[
\{\log\tilde\phhi^{(n)}_i, \log\tilde\phhi^{(n)}_j\} = \omega_{ij} - \omega_{i,(n-1)^2+1} + \omega_{j,(n-1)^2+1}
\]
 is constant on $\L_{nn}$ for any $i,j =1,\ldots, (n-1)^2$; we denote this constant by $\omega^{(n)}_{ij}$.
Furthermore, on $\L_{nn}$ we have $h_{ii}(Y) = a_{n+1,i}\cdots a_{n+1,n}$, $g_{ii}(X)=a_{1i}\cdots a_{1n}$, and so the 
log-canonicity of the entire family $\FFF_{nn}$ follows from the log-canonicity of $\bar\FFF_n$. By extension, we get 
the log-canonicity of the family $\FFF_{nn}$ on $\bar\L_{nn}$.

Using induction, assume that $\{\log\tilde\phhi^{(k)}_i, \log\tilde\phhi^{(k)}_j\} =\omega^{(k)}_{ij}$ is constant on 
$\bar\L_{kn}$ for any $i,j =1,\ldots, (k-1)(n-1)$. Substituting $(X,Y) \in \L_{k-1,n}$ into the 
$(k-1)(n-1)\times (k-1)(n-1)$ matrix $\Phi^{(k)}$ makes it reducible  with an irreducible 
$(k-2) (n-1) \times (k-2)(n-1)$ upper left block $\Phi^{(k-1)}$ and functions $\tilde\phhi^{(k)}_i$ restricted to 
$\L_{k-1,n}$ factor as $\tilde\phhi^{(k)}_i = \tilde\phhi^{(k-1)}_i \tilde\phhi^{(k)}_{(k-2)(n-1)+1}$
for $ i \leq (k-2) (n-1)$.  In addition, $\tilde\phhi^{(k)}_{(k-2)(n-1)+i}= a_{1,i+1}\cdots a_{1n}$ for $i=2,\ldots, n$.
Arguing precisely as above, we conclude  that
\begin{equation}
\label{band_const}
\{\log\tilde\phhi^{(k-1)}_i, \log\tilde\phhi^{(k-1)}_j\} = \omega^{(k)}_{ij} - \omega^{(k)}_{i,(k-2)(n-1)+1} + 
\omega^{(k)}_{j,(k-2)(n-1)+1}
\end{equation}
is a constant on $\L_{k-1,n}$ for any $i,j =1,\ldots, (k-2)(n-1)$ and denote it by $\omega^{(k-1)}_{ij}$. Therefore, 
the log-canonicity of the entire family $\FFF_{k-1,n}$ on $\L_{k-1,n}$, and hence on $\bar\L_{k-1,n}$, 
follows from the log-canonicity of $\FFF_{kn}$
on $\bar\L_{kn}$.
\end{proof}

\begin{proposition}
\label{compat-band}
The Poisson bracket \eqref{sklyadoubleXY} is compatible with the generalized cluster 
structure $\GCC(\Sigma_{kn})$ on $\L_{kn}$. 
\end{proposition}

\begin{proof} We will use induction again.  Within this proof, it will be convenient to refer to the vertex in
$\bar Q_{n}$ to which the variable $\phhi_{i}$ is attached and to the vertex in $Q_{nn}$ to which the variable 
$\tilde \phhi_{i}$ is attached as the vertex $i$ in the corresponding quiver.
Assume first that $k=n$, and let $y_i$ be the $y$-variable corresponding to 
the vertex $i$ in $\bar Q_{n}$ and $y_i^{(n)}$ be the $y$-variable corresponding to the vertex $i$ in $Q_{nn}$.
We claim that on $\L_{nn}$, $y_i^{(n)}=y_i$ for all mutable vertices  in $Q_{nn}$. 

Indeed, for  $n < i \leq (n-2)(n-1)$, the neighborhood of the vertex labeled $i$ in $Q_{nn}$ is identical to the 
neighborhood of the vertex labeled $i$ in $\bar Q_{n}$. We claim that on $\L_{nn}$, $y_i^{(n)}=y_i$. 

For $(n-2)(n-1) < i \leq (n-1)^2$, let $i'=i- (n-2)(n-1)$. Then 
\[
y_i^{(n)}=\frac{\tilde\phhi^{(n)}_{i+1}\tilde\phhi^{(n)}_{i-n}a_{1,i' +1}}{\tilde\phhi^{(n)}_{i-1}\tilde\phhi^{(n)}_{i-n+1}} = \frac{\phhi_{i+1} \phhi_{i-n}  g_{i'+1,i'+1}}  {\phhi_{i-1} \phhi_{i-n+1}  g_{i'+2,i'+2}} = y_i,
\] 
by \eqref{factor_n} and since $g_{jj}=a_{1j}\cdots a_{1n}$ on $\L_{nn}$.

For $1 < i \leq n$, 
\[
y_i^{(n)}=\frac{  \tilde\phhi^{(n)}_{i+1}   \tilde\phhi^{(n)}_{i+n-1}  a_{n+1,i}}  {\tilde\phhi^{(n)}_{i-1}   \tilde\phhi^{(n)}_{i+n}} = \frac{\phhi_{i+1} \phhi_{i+n-1}  h_{ii}}  {\phhi_{i-1} \phhi_{i+n}  h_{i+1,i+1}} = y_i,
\] 
by \eqref{factor_n} and since $h_{jj}=a_{n+1,j}\cdots a_{n+1,n}$ on $\L_{nn}$.

Finally, 
\[
y_1^{(n)} = \left (\frac{  \tilde\phhi^{(n)}_{2} }  {\tilde\phhi^{(n)}_{n+1}}\right )^n \frac{a_{n+1,1}} {\prod_{j=1}^n a_{1j} 
\left ( \prod_{j=2}^n a_{n+1,j} \right )^{n-1}} = \left (\frac{\phhi_{2} }  {\phhi_{n+1} h_{22}}\right )^n \frac{h_{11}} 
{g_{11}} = y_1.
\]
Therefore $\{\log y_i^{(n)}, \log  \tilde\phhi^{(n)}_{j}\} =  \{\log y_i,  \log{\phhi_j} - \log{\phhi_{(n-1)^2+1}}\} = d_i \delta_{ij}$, where $d_1=n$ and $d_i=1$ otherwise. The induction step is performed in precisely the same fashion by showing that on $\L_{k-1,n}$, for all mutable vertices in $Q_{k-1,n}$, $y_i^{(k-1)}=y_i^{(k)}$.
\end{proof}

\begin{remark}
When restricted to $\L_{kn}$, the Poisson structure \ref{sklyadoubleXY} coincides with the one considered in a recent 
paper \cite{Anton} on the space of  {\em properly bounded $n$-periodic difference operators}. A modification of that Poisson bracket for spaces of {\em sparse\/} pseudo difference operators was also considered in \cite{Anton} in order to derive complete integrability of a family of pentagram-like maps. It would be interesting to see if such a modification has a cluster-algebraic meaning as well.
\end{remark}

\subsection{Completeness}\label{bandcomplet}

The next two propositions are analogous  to Propositions \ref{tallmatrix}, \ref{longmatrix} and can be proved in exactly the same way. 

\begin{proposition}
\label{bandtallmatrix}
There exists a $(k-1)\times(k-1)$ unipotent upper triangular matrix $G=G(X, Y)$ such that

{\rm(i)} entries of $G$ are rational functions in $X$, $Y$ whose denominators are monomials in cluster variables 
$\tilde\phhi_{j n +1}$, $j=1,\ldots, k-2$, and 

{\rm(ii)} the $(2k-2)\times (k-1)$ matrix $S=\begin{bmatrix} Y^{[n-k+2,n]}_{[n-k+2,n]}\\ G X_{[1,k-1]}^{[n-k+2,n]}\end{bmatrix}$ satisfies
\begin{equation}
\label{bandtallmatrixminors}
 \det G_{[ k+j-i , k+j-1]}^{[k-i , k-1]} = \frac{\tilde\phhi_{j n - i +1 }} {\tilde\phhi_{j n +1}}, 
\qquad j=1,\ldots, k-2,\quad i=1,\ldots, k-1.
\end{equation}
\end{proposition}

\begin{proposition}
\label{bandlongmatrix}
There exists a $(k-1)\times (k-1)$ unipotent lower triangular matrix $H=H(X, Y)$ such that

{\rm (i)} entries of $H$ are rational functions in $X$, $Y$ whose denominators are monomials in cluster variables 
$\tilde\phhi_{j (n-1) +1}$, $j=1,\ldots, k-2$, and 

{\rm (ii)} the $(n-1)\times (n+k-1)$ band matrix 
\[
T=\begin{bmatrix} X_{[2,n]}^{[n-k+2,n]} & Y_{[2,n]}^{[1,n-k+1]} & Y_{[2,n]}^{[n-k+2,n]}H   \end{bmatrix}
\] 
satisfies
\begin{equation}
\label{bandlongmatrixminors}
 \det T_{[ n-i, n-1]}^{[n+j-i+1 , n+j]} = \frac{\tilde\phhi_{(n-j) (n -1) - i +1 }} {\tilde\phhi_{(k-j) (n-1) +1}}, 
\qquad j=2,\ldots, k-1, \quad i=1,\ldots, n-1.
\end{equation}
\end{proposition}

The completeness statement for $\GCC(\L_{kn})$ is based on the following result.

\begin{proposition}
\label{bandmatrixentries}
All matrix entries $a_{ij}$, $i=2,\ldots, k$, $j=2,\ldots n$, are cluster variables in $\GCC(\L_{kn})$.
\end{proposition}

\begin{proof} We use an argument similar to that in the proof of Proposition \ref{matrixentries}. Namely, we will temporarily freeze certain subsets of vertices in $Q_{kn}$ and compare the result with initial seeds of appropriate previously studied cluster structures. First, freeze the vertices in the top row of $Q_{kn}$, that is, those that correspond to 
$\tilde\phhi_{j (n-1) +1}$, $j=1,\ldots, k-2$. Then the vertices that correspond to $a_{11}$, $a_{12}$, $a_{k+1,1}$, 
$a_{k+1,n}$ become isolated. The subquiver $\tilde Q_{kn}$ of $Q_{kn}$ formed by the remaining vertices is closely related to an initial quiver $Q'_{k,n-1}$ for the regular cluster structure  $\CC(\L'_{k,n-1})$ on the space $\L'_{k,n-1}$ of 
$(n-1)\times (n+k-1)$ band matrices with $k+1$ diagonals that was constructed in \cite[Section 10]{quasichris}  
 via a quasi-isomorphism from the regular
cluster structure on the affine cone $\widehat{\Gr}(n-1,n+k-1)$
over the Grassmannian $\Gr(n-1,n+k-1)$. The difference is that in  $Q'_{k,n-1}$ there are no edges between the vertices in the top row and vertices in the bottom row. Let $T$ be an element in $\L'_{k,n-1}$:
\[
T= \left [
\begin{array}{ccccccc}
t_{12} & t_{22} & \cdots &t_{k+1,2} & 0 & \cdots & 0\\
0 & t_{13} & t_{23}  & \cdots & t_{k+1,3} & 0 & \vdots \\
\vdots & \ddots & \ddots & \ddots & \cdots& \ddots & \ddots\\
0 & \cdots &0 & t_{1n} & t_{2n} &\cdots & t_{k+1,n}
\end{array}
\right ];
\]
note that $t_{ij} = T_{j-1, i+j-2}$. Initial cluster variables that correspond to $Q'_{k,n-1}$ in $\CC(\L'_{k,n-1})$ are functions $\psi_{ij}(T)$, $i=2,\ldots, k$, $j=2,\ldots, n$, where $\psi_{ij}(T)$ is the maximal dense minor of $T$ with 
$t_{ij}$ in the upper left corner, and matrix entries $t_{1j}$, $j=3,\ldots, n$, $t_{k+1,j}$, $j=2,\ldots, n-1$.
The latter are  frozen and attached to the same vertices in $Q'_{kn}$  that $a_{1j}$, $j=3,\ldots, n$, and 
$a_{k+1,j}$, $j=2,\ldots, n-1$, are attached in $\tilde Q_{kn}$. The function $\psi_{ij}(T)$ is  attached to the 
vertex $(i-1,j)$ of  $Q'_{kn}$. In addition to the frozen variables mentioned above, the variables  attached to 
the vertices of  the first row of $Q'_{kn}$ are also frozen. {\em All\/} the irreducible row-dense minors of $T$ are cluster  variables in $\CC(\L'_{k,n-1})$.

Note that in \cite{quasichris} an initial seed for $\CC(\L'_{k,n-1})$ is not described explicitly. To justify our explicit description of the seed above we rely on two observations. First, the functions $\psi_{ij}(T)$ are images under the quasi-isomorphism of  \cite{quasichris} of cluster variables of the initial seed for ${\widehat \Gr}(n-1,n+k-1)$ as described in \cite[Chapter 4]{GSVb}. This means that subquivers formed by non-frozen vertices in  the initial quivers for these two structures coincide, the only difference is in the arrows that connect  frozen and  non-frozen variables. Second, the edges between frozen and non-frozen variables in the initial quiver for  $\CC(\L'_{k,n-1})$ are uniquely determined by the regularity of this cluster structures. To see this, one needs to analyze, in a bottom to top order, the exchange relations for left-most and right-most mutable vertices in $Q'_{k,n-1}$ and apply the standard Desnanot--Jacobi identities.

Now, assume that $T\in \L'_{k,n-1}$ is the matrix defined in Proposition \ref{bandlongmatrix}. Then it follows from 
\eqref{bandlongmatrixminors} that $t_{1j}=a_{1j}$ for $j=3,\ldots, n$, $t_{k+1,j}=a_{k+1,j}$ for $j=2,\ldots, n-1$ and
 $\psi_{ij}(T) = \frac{\tilde\phhi_{(k-i) (n -1) + j -1 }} {\tilde\phhi_{(k-i+1) (n-1) +1}}$ for $i=2,\ldots, k$, 
$j=2,\ldots, n$. Then, just like in the proof of Proposition \ref{matrixentries}, we conclude that sequences of mutations 
in $\GCC(\L_{kn})$ that do not involve functions $\tilde\phhi_{j (n-1) +1}$, $j=1,\ldots, k-1$, result in corresponding sequences of mutations in $\CC(\L'_{k,n-1})$ with the initial seed defined by $Q'_{kn}$ and functions 
$t_{1j}$, $t_{k+1,j}$, $\psi_{ij}(T)$. Since every $t_{ij}$ is a cluster variable in $\CC(\L'_{k,n-1})$ and $t_{ij} = a_{ij}$ unless $(i,j) \in R= \{ (l,m), \ l=3,\ldots, k,\  m = n-l+3, \ldots, n \}$, we use the argument in the proof of 
Proposition \ref{matrixentries} to conclude that, for $(i,j) \notin R$, $a_{ij}$ is a cluster variable in $\GCC(\L_{kn})$.

The case of $(i,j) \in R$ is treated in a similar way. Namely, consider the vertices corresponding to 
$\tilde\phhi_{j n +1}$, $j=0,\ldots, k-2$, and to $\tilde\phhi_{(k-2)(n-1)-jn}$, $j=0,\ldots,k-3$, 
in $Q_{kn}$. The vertices in the first set form an anti-diagonal that starts in the upper right corner 
of the grid formed by non-frozen vertices of $Q_{kn}$, and the vertices in the second set lie
immediately below the anti-diagonal that starts in the lower left corner of this grid. 
Let us temporarily freeze  the vertices in both sets as well as all the vertices lying between them. 
The quiver $\hat Q_{kn}$ obtained by deleting all isolated vertices is, once again, similar to the quiver of the 
initial seed for the cluster structure $\CC(\L'_{2k-2,k-1})$ on  a set of finite band matrices, 
this time $(2k -2) \times (k-1)$ matrices with $k$ diagonals. To see this, one just needs to move the vertices of 
$\hat Q_{kn}$ around in a way illustrated in Figure \ref{Q57_red} for the case $n=7, k=5$.

\begin{figure}[ht]
\begin{center}
\includegraphics[width=8cm]{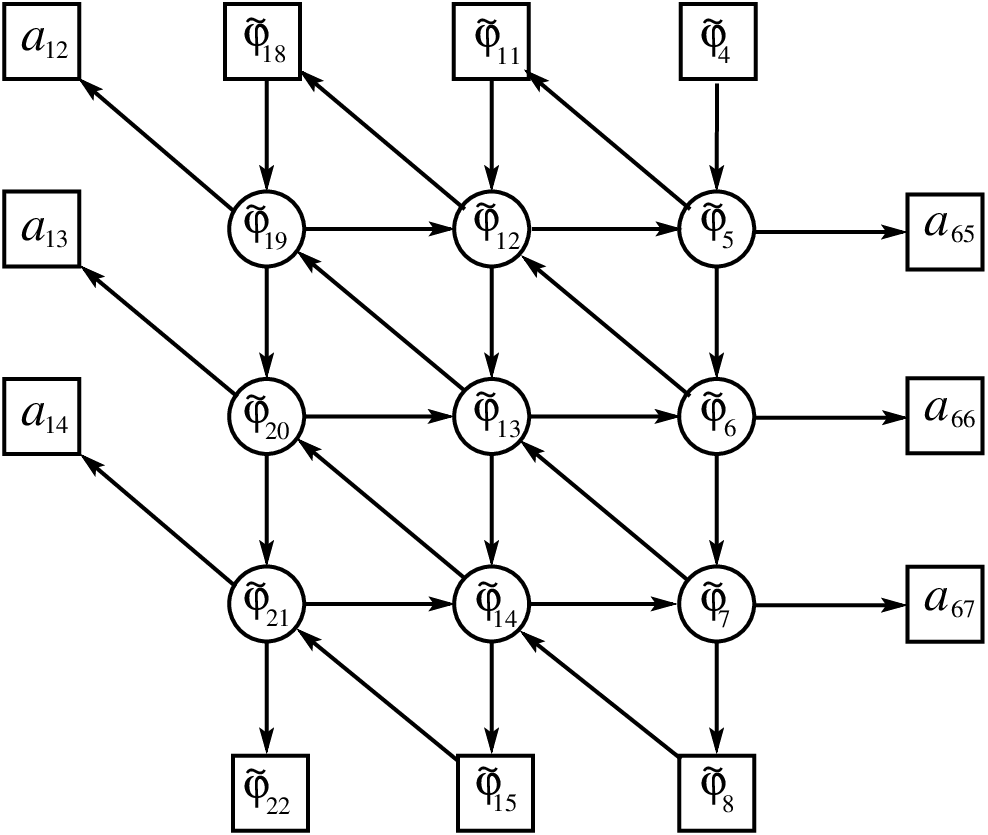}
\end{center}
\caption{Quiver $\hat Q_{57}$}
\label{Q57_red}
\end{figure}

The latter quiver is isomorphic to $Q'_{k-1,2k-2}$, but we denote it by $Q^T_{k-1,2k-2}$ to reflect the fact  that the initial seed for  $\CC(\L'_{2k-2,k-1})$ can be easily obtained from the one for $\CC(\L'_{k-1, 2k-2})$ via transposition. 
To obtain $Q^T_{k-1,2k-2}$ from $\hat Q_{kn}$  one simply needs to erase vertices in the bottom row that correspond 
to functions $\tilde\phhi_{j n +1}$, $j=1,\ldots, k-2$. Then the functions 
\[
\frac{\tilde\phhi_{j n - i +1 }} {\tilde\phhi_{j n +1}}, \qquad j=1,\ldots, k-2,\quad i=1,\ldots, k-1
\]
 are subject to exchange relations in  $\CC(\L'_{2k-2,k-1})$. At the same time, by Proposition~\ref{bandtallmatrix}, these functions represent the minors of  the band matrix $A$ that, together with the frozen variables
$a_{1i}$, $i=2,\ldots,k-1$, and $a_{k+1,i}$, $i=n-k+3,\ldots,n$, form an initial seed for 
$\CC(\L'_{2k-2,k-1})$. Then the argument concludes exactly as above.

\end{proof}

To establish completeness of $\GCC(\L_{kn})$, it now remains to show that $a_{i1}$, $i=2,\ldots, k$, belong to the  generalized upper cluster algebra $\UU(\GCC(\L_{kn}))$.  Since these are the entries of the first row of $X$ as defined in \eqref{XYband}, one applies a modification of  
Lemma \ref{firstrow} and its proof. To this end, we replace the system \eqref{TheSystem} with 
\begin{equation}
\label{TheBandSystem}
\tilde c_j(X,Y) = \tilde c_j(\bar X, Y)  + \sum_{i=1}^k a_{i1} z_{ij}(\bar X, Y), 
\end{equation}
where $\bar X$ is defined as in the proof of Lemma \ref{firstrow}. The implication $\det Z=\det\left (z_{ij}(\bar X, Y)\right )_{i,j=1}^k=0 \Rightarrow \tilde\phhi_1=0$ is established via the same reasoning as before, except now $XY^{-1} =\begin{bmatrix} U & \star \\ 0 & 0\end{bmatrix}$, where the block $U$ is $k\times k$ and \eqref{Schur} implies that $ \det \left [U^{k-1}e_1 U^{k-2}e_1 \cdots U e_1\ e_1 \right ]=0$. As before,  \cite[Lemma 3.3]{GSVstaircase} states that the determinant in the last equation is a nonzero multiple of $\tilde\phhi_1$ and the desired implication is confirmed. The rest of the arguments in the proof of Lemma \ref{firstrow} transfer to the current situation in a straightforward way.

\section*{Acknowledgments}

Our research was supported in part by the NSF research grant DMS \#1702054 (M.~G.), NSF research grant DMS \#1702115 
and International Laboratory of Cluster Geometry NRU HSE, RF Government grant, ag.~\# 075-15-2021-608 from 08.06.2021 (M.~S.), 
and ISF grant \#1144/16 (A.~V.). While working on this project, we benefited from support of the following institutions 
and programs: Research Institute for Mathematical Sciences, Kyoto (M.~G., M.~S., A.~V., Spring 2019), Research in Pairs Program at the Mathematisches Forschungsinstitut Oberwolfach (M.~S. and A.~V., Summer 2019), Istituto Nazionale di Alta Matematica Francesco Severi and the Sapienza University of Rome (A.~V., Fall 2019), University of Haifa 
(M.~G., Fall 2019), Mathematical Science Research Institute, Berkeley (M.~S., Fall 2019),
Michigan State University (A.~V., Spring 2020), University of Notre Dame (A.~V., Spring 2020). 
We are grateful to all these institutions for their hospitality and outstanding working conditions they provided. Special thanks are due to Peigen Cao and Fang Li  who in response to our request provided a generalization \cite{CL2} of their previous results, to Linhui Shen for pointing out to us the preprint \cite{Zhou}, to Alexander Shapiro and Gus Schrader for many fruitful discussions, and to the reviewers for constructive suggestions.

\end{document}